\documentclass{amsart}
\numberwithin{equation}{section}
\usepackage[a4paper,top=2.5cm,bottom=2.5cm,left=2.5cm,right=2.5cm]{geometry}
\usepackage[utf8]{inputenc}

\title{Metric completions of discrete cluster categories}
\author{Charley Cummings and Sira Gratz}

\usepackage[dvipsnames]{xcolor}
\usepackage{amsmath}
\usepackage{amssymb}
\usepackage{amsthm}
\usepackage{enumerate} 
\usepackage{enumitem}
\usepackage{float}
\usepackage{hyperref}
\usepackage{tikz}
\usetikzlibrary{arrows,decorations.pathmorphing,decorations.markings,decorations.pathreplacing,cd}

\DeclareFontFamily{U}{min}{}
\DeclareFontShape{U}{min}{m}{n}{<-> udmj30}{}
\newtheorem{theorem}{Theorem}[section]
\newtheorem{proposition}[theorem]{Proposition}
\newtheorem{corollary}[theorem]{Corollary}
\newtheorem{lemma}[theorem]{Lemma}
\theoremstyle{definition}
\newtheorem{definition}[theorem]{Definition}

\newtheorem{remark}[theorem]{Remark}

\DeclareMathOperator{\Ab}{Ab}

\DeclareMathOperator{\Mod}{Mod}

\DeclareMathOperator{\cone}{cone}
\DeclareMathOperator{\Hom}{Hom}
\DeclareMathOperator{\Ext}{Ext}
\DeclareMathOperator{\id}{id}
\DeclareMathOperator{\colim}{colim}
\DeclareMathOperator{\add}{add}

\DeclareMathOperator{\cx}{Conv}
\DeclareMathOperator{\mocolim}{mocolim}

\newcommand{\aisle}[0]{\mathcal{X}}
\newcommand{\raisle}[0]{{\aisle_R}}
\newcommand{\coaisle}[0]{\mathcal{Y}}
\newcommand{\rcoaisle}[0]{{\coaisle_R}}

\newcommand{\yo}{\!\text{\usefont{U}{min}{m}{n}\symbol{'207}}\!}

\newcommand{\arc}[3]{\{x_{#1}^-,x_{#2}^{({#3})}\}}
\newcommand{\fromarc}[1]{\cE^+({#1})}
\newcommand{\toarc}[1]{\cE^-({#1})}

\newcommand{\cA}{\mathcal{A}}
\newcommand{\cB}{\mathcal{B}}
\newcommand{\cC}{\mathcal{C}}
\newcommand{\cD}{\mathcal{D}}
\newcommand{\cE}{\mathcal{E}}
\newcommand{\cM}{\mathcal{M}}
\newcommand{\cP}{\mathcal{P}}
\newcommand{\cQ}{\mathcal{Q}}
\newcommand{\cR}{\mathcal{R}}
\newcommand{\cS}{\mathcal{S}}
\newcommand{\cT}{\mathcal{T}}

\newcommand{\cY}{\mathcal{Y}}
\newcommand{\cZ}{\mathcal{Z}}

\newcommand{\fR}{\mathfrak{R}}
\newcommand{\fS}{\mathfrak{S}}

\newcommand{\bZ}{\mathbb{Z}}
\newcommand{\bN}{\mathbb{N}}
\newcommand{\bK}{\mathbb{K}}

\newcommand{\rx}{\mathrm{x}}
\newcommand{\ry}{\mathrm{y}}

\tikzcdset{column sep= {7em,between origins},row sep= {6em,between origins}}

\begin{document}

\begin{abstract}
    Neeman shows that the completion of a triangulated category with respect to a good metric yields a triangulated category. We compute completions of discrete cluster categories with respect to metrics induced by internal t-structures. In particular, for a coaisle metric this yields a new triangulated category which can be interpreted as a topological completion of the associated combinatorial model. Moreover, we show that the completion of any triangulated category with respect to an internal aisle metric is a thick subcategory of the triangulated category itself.
\end{abstract}

\setlist[enumerate,1]{label={(\roman*)}}

\maketitle
\tableofcontents

\section{Introduction}
Despite the ubiquity and popularity of triangulated categories across mathematical fields, processes for generating new triangulated categories from old are notoriously few and far between.  We may, of course, take a triangulated subcategory or a Verdier quotient, but beyond that, what tools do we have available? There is a small, but steadily growing arsenal of astonishingly modern, sophisticated and elaborate techniques including orbit categories \cite{Keller_orbit}, modules over a separable monad \cite{Balmer_monad}, relative stable categories \cite{Beligiannis_relative}, \cite{Balmer-Stevenson_relative} and, as of very recently, completions \cite{Krause_completing}, \cite{Neeman--2018--DetermineEachOther}, \cite{neeman2020metricssurvey}. 

We follow Neeman’s approach from \cite{neeman2020metricssurvey} which utilises metrics on a triangulated category to form a completion which is again a triangulated category. These metrics can be realised as metrics in the spirit of \cite{lawvere1973metrics_on_categories}: Morphisms get assigned a length and a triangle inequality is satisfied---the length of a composition of two morphisms cannot exceed the sum of their individual lengths. Given a metric $\cM$ on a triangulated category, following our intuition from metric spaces, we now do what we set out to do: We add in objects which arise as ``limit points'' of Cauchy sequences, that is, of $\bN$-shaped diagrams whose morphisms get increasingly smaller. These ``limit points'' are in fact colimits in the ind-completion of our triangulated category, and restricting to well-behaved ones, that is those which treat sufficiently small morphisms as isomorphisms, gives rise to a subcategory $\mathfrak{S}_\cM$ of the ind-completion. Neeman \cite{Neeman--2018--DetermineEachOther} shows that $\mathfrak{S}_\cM$ is a triangulated category with triangulated structure induced by the triangulated structure on our initial category. Interesting examples abound; for example one can realise the bounded derived category of finitely generated modules over a noetherian ring as a completion of its category of perfect complexes, and vice versa.

Completions provide an entirely novel way to, in principle, cook up a plethora of new examples of triangulated categories. However, so far in practice, this has not happened. The completions constructed in Neeman’s work are all computed within a suitable, already known, ambient triangulated category, a so-called good extension of our initial category. This has also been exploited in a representation theoretic context in \cite{cyril}, which includes the computation of all possible completions of the bounded derived category of a hereditary finite dimensional algebra with finite representation type. In this paper we present a, to our knowledge first, non-trivial computation that does completely without the aid of a known ambient triangulated category: completions of discrete cluster categories of type $A$ with respect to metrics coming from t-structures. A discrete cluster category is an algebraic triangulated category, so a good extension to an ambient triangulated category exists in theory. However in practice, beyond the simplest case (treated holistically in \cite{ACFGSII}) we do not have an explicit description thereof. In this simplest case, \cite{Fisher} explicitly computes an enlargement adding Pr\"ufer objects. Amazingly, even though the paper predates the metric completion techniques, these turn out to be precisely the homotopy colimits of compactly supported Cauchy sequences, in a good extension, with respect to the (up to equivalence unique) non-trivial coaisle. 

Without the help of an explicit good extension in the general case, we instead utilise the combinatorial model employed in \cite{ITcyclic} to construct these discrete cluster categories---a disc with an infinite, discrete set $\cZ$ of marked points on its boundary, with finitely many two-sided accumulation points. The discrete cluster category $\cC(\cZ)$ is a $\bK$-linear (over some field $\bK$) Krull-Schmidt  triangulated category, whose indecomposable objects correspond to diagonals with endpoints in $\cZ$ (called arcs of $\cZ$), and morphisms between arcs are described by angles between the relevant diagonals.

We are specifically interested in metrics coming from t-structures on $\cC(\cZ)$. More generally, given a t-structure $(\aisle,\coaisle)$ on a triangulated category $\cT$ one can define an associated aisle metric $\cM_\aisle$. In this metric, a morphism is short, say at most length $\frac{1}{n+1}$, if its cone lies in the $n$-th suspension of $\aisle$. Dually, one defines an associated coaisle metric $\cM_\coaisle$, which considers a morphism to be of length at most $\frac{1}{n+1}$, if its cone lies in the $n$-th desuspension of $\coaisle$. The case of aisle metric completions is quickly settled:

\begin{theorem} [Corollary~{\ref{C:aisle_in_metric_completion}}]
    Let $(\aisle,\coaisle)$ be a t-structure on a triangulated category $\cT$ and consider the aisle metric $\cM_\aisle$. Then we have
    \begin{equation*}
        \fS_{\cM_\aisle} \simeq \add \bigcup_{p\in\bZ} \Sigma^p\coaisle.
    \end{equation*}
In particular, the completion $\fS_{\cM_\aisle}$ is triangulated equivalent to a thick subcategory of $\cT$.
\end{theorem}

This observation also appears in the context of extendable t-structures, cf.\ \cite{BCMPZ-completion}. Note that, while we admittedly still land inside a known triangulated category (the initial category itself), we do not make use of a good extension to obtain the result. New and fun triangulated categories are thus not to be come by using aisle metrics from internal t-structures. However, coaisle metrics prove to be considerably more fruitful. We explicitly compute the completion of $\cC(\cZ)$ with respect to any coaisle metric and, in general, obtain an a priori completely new triangulated category. To do so we exploit the combinatorial model. The new objects we add are certain arcs that end at one or two accumulation points. Depending on our choice of coaisle metric we may choose which accumulation points are added as potential endpoints and, to a lesser extent, where the arcs incident with this point are allowed to go.

More concretely, we use the classification from \cite{Gratz--Zvonareva--2023--tstructures_in_clus_cat} of t-structures on $\cC(\cZ)$ in terms of $\overline{\cZ}$-decorated non-crossing partitions: Assume that $\cZ$ has $N$ proper accumulation points, that is, limit points with respect to the standard topology which do not lie in $\cZ$. Then every t-structure $(\aisle,\coaisle)$ corresponds to a a $\overline{\cZ}$-decorated non-crossing partition, i.e.\ a pair $(\cP,\rx)$, where $\cP$ is a non-crossing partition of the set $[N] = \{1, \ldots, N\}$ under the natural linear order, and $\rx$ is a $\overline{\cZ}$-decoration of $\cP$, i.e.\ an $N$-tuple from $\overline{\cZ}$ satisfying mild additional conditions (cf.\ Definition~\ref{D:decoratednc}). This description of t-structures allows for a concrete way to describe the aisle, and coaisle, in terms of the combinatorial model (cf.\ Theorem~\ref{T:GZ_classification_aisles}). Given a block $B$ of $\cP$, we define a subset $\overline{\cZ}_B$ of $\overline{\cZ}$, which is determined by $B$ and $\rx$, and which in general contains accumulation points in $\overline{\cZ} \setminus \cZ$. An arc of $\overline{\cZ}_B$ is a diagonal in the disc with endpoints in $\overline{\cZ}_B$.

\begin{theorem}[{Theorem~\ref{T:combinatorial description of completion}}]
    Let $(\aisle,\coaisle)$ be a t-structure on $\cC(\cZ)$ corresponding to a $\overline{\cZ}$-decorated non-crossing partition $(\cP,\rx)$. The metric completion $\mathfrak{S}_{\cM_\coaisle}$ of $\cC(\cZ)$ with respect to the coaisle metric $\cM_\coaisle$ is equivalent to
    \[
        \mathfrak{S}_{\cM_\coaisle}(\cC(\cZ)) \simeq \bigoplus_{B \in \cP} \mathfrak{S}_{\cM_\coaisle}(B),
    \]
    where each $\mathfrak{S}_{\cM_\coaisle}(B)$ can be described combinatorially as follows:
    \begin{itemize}
        \item The indecomposable objects of $\mathfrak{S}_{\cM_\coaisle}(B)$ are in one-to-one correspondence with arcs of $\overline{\cZ}_B$.
        \item Let $F$, $G$ and $H$ be arcs of $\overline{\cZ}_B$. Then we have $\dim_{\bK} (\Hom_{\mathfrak{S}_{\cM_\coaisle}(B)}(F, G)) \leq 1$. Moreover, whether or not a non-trivial morphism from $F$ to $G$ exists, and factors through $H$, can be read from the relative positioning of the arcs $F$, $G$ and $H$.
    \end{itemize}
Furthermore, $\mathfrak{S}_{\cM_\coaisle}$ is a triangulated category with suspension functor $\Sigma$ acting on arcs by a one-step clockwise rotation. Its distinguished triangles are diagrams which are isomorphic to colimits in $\Mod \cC(\cZ)$ of compactly supported Cauchy sequences of triangles in $\cC(\cZ)$. In particular, for each of the configurations for $U$, $W$, and $V=V_1 \oplus V_2$ as in Figure~\ref{F:cone} we obtain a triangle
\[
    U \to V \to W \to \Sigma U.
\]
\end{theorem}

\begin{figure}[H]
    \centering
    \includegraphics[scale=0.8]{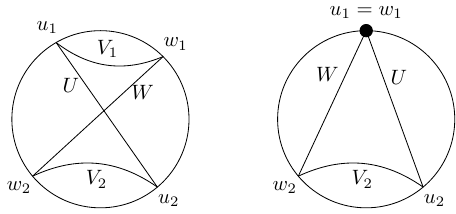}
    \caption{Representation of distinguished triangles between indecomposable objects in the completion $\fS_{\cM_\coaisle}$ of $\cC(\cZ)$. The endpoints of the arcs lie in $\overline{\cZ}_B \subseteq \overline{\cZ}$ for some block $B$. On the right-hand side, we have $u_1 = w_1 \in \overline{\cZ} \setminus \cZ$ and $V_1 = 0$.}
    \label{F:cone}
\end{figure}

Discrete cluster categories of type $A$ have received a lot of well-deserved attention in recent years (\cite{GHJ}, \cite{Murphy_Orlov}, \cite{Murphy_K0}, \cite{CKP_discrete}, \cite{Franchini_torsion_pairs}). In particular, Paquette and Y{\i}ld{\i}r{\i}m \cite{Paquette--Yildirim--2021:combinatorial_completion} were interested in the combinatorial process of adding arcs ending in accumulation points, and explicitly constructed, as a certain Verdier quotient, a triangulated category which contains these. Astonishingly, in the special case where we pick our t-structure to be non-degenerate and bounded above (such a t-structure always exists in $\cC(\cZ)$), our completion is, as a category, equivalent to the category constructed in \cite{Paquette--Yildirim--2021:combinatorial_completion}. While this equivalence commutes with the autoequivalence of the respective triangulated structure and preserves triangles with two indecomposable terms, we do not know if it is in fact a triangulated equivalence. This remains to be investigated, and we encourage the interested reader to do so. 

The discrete cluster category $\cC(\cZ)$ is self-dual, that is, equivalent to its opposite category. In particular, aisles in $\cC(\cZ)$ correspond to coaisles in $\cC(\cZ)^\mathrm{op} \simeq \cC(\cZ)$, and an aisle metric on $\cC(\cZ)$ has a dual coaisle metric on $\cC(\cZ)^\mathrm{op} \simeq \cC(\cZ)$. Nevertheless, as we show in this paper, the completions with respect to aisle and coaisle metrics look drastically different. This should not surprise; already the ind-completions in which all computations take place fail to behave well with taking opposite categories, as in general $\Mod(\cT)^\mathrm{op}$ is not equivalent to $\Mod(\cT^\mathrm{op})$.

\subsection*{Acknowledgements}
This work was supported by VILLUM FONDEN  (Grant Number VIL42076) and EPSRC (Grant Number EP/V038672/1).

\section{Conventions}
All our subcategories are full and replete, even where not explicitly stated. If $\cA$ is a subcategory of a category $\cC$, then we write $\cA \subseteq \cC$, and if $A$ is an object in $\cA$, then we write $A \in \cA$. For a set of objects $S$ in an additive category $\cT$, we denote by $\add(S)$ the full subcategory of $\cT$ whose objects are summands of finite direct sums of objects in $S$. For subcategories $\cA$ and $\cB$ of $\cT$, we define subcategories:
\begin{enumerate}
    \item $\cA * \cB = \add\{T \in \cT \mid \exists \text{ triangle } A \to T \to B \to \Sigma A \text{ such that $A \in \cA$ and $B \in \cB$}\}$;
    \item $\cA \cap \cB$ to be the full subcategory of $\cT$ with objects $T \in \cT$ such that $T \in \cA$ and $T \in \cB$;
    \item $\cA \cup \cB$ to be the full subcategory of $\cT$ with objects $T \in \cT$ such that $T \in \cA$ or $T \in \cB$.
\end{enumerate}

Throughout, we work over a fixed field $\bK$ and we denote by $D$ the duality $\Hom_\bK(-,\bK)$.

\section{Preliminaries}
In this section, we fix a triangulated category $\cT$ with suspension functor $\Sigma \colon \cT \rightarrow \cT$.

\subsection{Preliminaries}

We recall some fundamental definitions from the theory of metrics and completions of triangulated categories. We mostly follow conventions from \cite{neeman2020metricssurvey}. Where convenient, we introduce additional terminology that will be useful throughout the rest of the paper.

\begin{definition}\label{D:metric}
    A {\em metric} on $\cT$ is a collection of full subcategories $\cM = \{B_t\}_{t \geq 0}$ of $\cT$ with $B_0 = \cT$ such that, for all $t \geq 0$, the following hold:
    \begin{itemize}
        \item[\textbf{M0}]{$0 \in B_t$;}
        \item[\textbf{M1}]{$B_{t+1} \subseteq B_t$.}
    \end{itemize}
    A {\em good metric} on $\cT$ is a metric $\cM = \{B_t\}_{t \geq 0}$ on $\cT$ such that, for all $t \geq 0$, the following holds:
    \begin{itemize}
        \item[\textbf{M2}]{$B_t * B_t \subseteq B_t$, i.e.\ $B_t$ is extension closed;}
        \item[\textbf{M3}]{$\Sigma B_{t+1} \subseteq B_{t}$ and $\Sigma^{-1} B_{t+1} \subseteq B_{t}$}.
    \end{itemize}
\end{definition}

Axiom \textbf{M2} ensures that the metric is {\em non-Archimedian} and in \cite{Neeman--2018--DetermineEachOther} is taken to be part of the definition of a metric (cf.\ the discussion before \cite[Definition~10]{Neeman--2018--DetermineEachOther}). Definition~\ref{D:metric} yields a metric on a category in the spirit of Lawvere \cite{lawvere1973metrics_on_categories}. Explicitly, given a metric $\cM = \{B_t\}_{t \geq 0}$ on $\cT$, we assign to a morphism $f \colon A \to B$ the length:
    \[
        \mathrm{length}(f) = \inf \left\{\frac{1}{t+1} \mid \cone(f) \in B_t\right\}.
    \]
For more details, we refer the reader to the exposition in \cite{neeman2020metricssurvey}.

\begin{definition}
    A {\em sequence $\mathbf{E} = (E_n,f_{n})_{n>0}$ in $\cT$} is an $\bN$-shaped diagram
    \begin{center}
		\begin{tikzcd}[column sep= {7em,between origins},row sep= {5em,between origins}]
			E_1 \arrow[r, "f_1"] & E_2 \arrow[r, "f_2"] & \ldots \arrow[r] \arrow[r, "f_{n-1}"] & E_n \arrow[r,"f_n"] & \ldots
		\end{tikzcd}
	\end{center}
    in $\cT$. For all integers $0<m < m'$, we denote by $f_{m,m'}$ the composition
    \begin{equation*}
        f_{m'} \circ \cdots \circ f_{m+1} \circ f_{m} \colon E_m \rightarrow E_{m'}.
    \end{equation*}    
    A sequence $\mathbf{E} = (E_n,f_n)_{n>0}$ in $\cT$ is a {\em null sequence}, if for all integers $m > 0$ there exists an integer $m' > m$ with $f_{m,m'} = 0$.
\end{definition}

\begin{definition}
    Fix a good metric $\cM = \{B_t\}_{t \geq 0}$ on $\cT$. A {\em Cauchy sequence with respect to $\cM$} is a sequence $\mathbf{E} = (E_n,f_{n})_{n >0}$ in $\cT$ such that for all $t \geq 0$ there exists an $n_t > 0$ such that for all $m' > m \geq n_t$ the cone of $f_{m,m'}$ lies in $B_t$.
\end{definition}

Remembering that a good metric measures the lengths of morphisms, we see that a Cauchy sequence is aptly named. It is a sequence $(E_n,f_n)_{n>0}$ in $\cT$ such that for all $\varepsilon >0$ there exists an $n_\varepsilon$ such that for all $m' > m \geq n_\varepsilon$ we have $\mathrm{length}(f_{m,m'}) < \varepsilon$.

Consider the category $\Mod\cT$ of contravariant additive functors from $\cT$ to the category $\Ab$ of abelian groups. Recall that $\cT$ may be realised as a full subcategory of $\Mod\cT$ under the Yoneda-embedding $\yo \colon \cT \to \Mod\cT$ that maps an object $A \in \cT$ to the contravariant functor $\Hom_{\cT}(-,A)$.

\begin{definition}
    Let $\mathbf{E} = (E_n,f_n)_{n>0}$ be a sequence in $\cT$. Its image under the Yoneda embedding $\yo$ is a sequence $\yo(\mathbf{E}) = (\yo(E_n), \yo(f_n))_{n>0}$ in $\Mod \cT$. The {\em module colimit} of $\mathbf{E}$, denoted  by $\mocolim\mathbf{E}$, is the colimit of $\yo(\mathbf{E})$ in $\Mod \cT$:
    \[
        \mocolim \mathbf{E} = \colim \yo(\mathbf{E}).
    \]
\end{definition}

\begin{definition}
    Let $\cM$ be a good metric on $\cT$. The {\em completion $\mathfrak{L}_{\cM}(\cT)$ of $\cT$ along Cauchy sequences with respect to $\cM$}  is the full subcategory of $\Mod\cT$ whose objects are the module colimits of Cauchy sequences in $\cT$ with respect to $\cM$.
\end{definition}

The completion of a triangulated category with respect to a fixed good metric in the sense of Neeman \cite{Neeman--2018--DetermineEachOther} is the restriction of the completion along Cauchy sequences to compactly supported objects.

\begin{definition}
    Let $\cM = \{B_t\}_{t \geq 0}$ be a good metric on $\cT$ and fix an integer $t \geq 0$.  A functor $F \in \Mod \cT$ is {\em compactly supported at $t$ with respect to $\cM$} if, for all morphisms $f$ in $\cT$ with $\cone(f) \in B_t$, the morphism $F(f)$ is an isomorphism. A functor is {\em compactly supported with respect to $\cM$} if it is compactly supported at some $t \geq 0$. We denote by $\mathfrak{C}_\cM$ the full subcategory of $\Mod\cT$ whose objects are compactly supported with respect to $\cM$.

    A sequence in $\cT$ is {\em compactly supported (at $t$) with respect to $\cM$} if its module colimit is compactly supported (at $t$) with respect to $\cM$.
\end{definition}

If the good metric $\cM$ is clear from context, we usually omit the qualifier ``with respect to $\cM$'' when discussing compactly supported functors or sequences.
Rephrasing loosely in terms of lengths of morphisms, a functor is compactly supported if and only if it takes sufficiently small morphisms to isomorphisms.

\begin{remark}
     Since module colimits of sequences are cohomological functors, a sequence $\mathbf{E}$ in $\cT$ is compactly supported at $t$, with respect to a good metric $\cM = \{B_t\}_{t \geq0}$, if and only if for all objects $A \in B_t$ we have $\mocolim\mathbf{E} (A) = 0$. 
\end{remark}

\begin{definition}
    Let $\cM$ be a good metric on $\cT$. The {\em completion of $\cT$ with respect to $\cM$} is the full subcategory
    \[
        \mathfrak{S}_\cM = \mathfrak{L}_\cM \cap \mathfrak{C}_\cM
    \]
    of $\Mod \cT$ that consists of the module colimits of compactly supported Cauchy sequences in $\cT$ with respect to the metric $\cM$.
\end{definition}

\begin{theorem}\cite[Theorem~15]{neeman2020metricssurvey}
    For any good metric $\cM$ on $\cT$, the category $\fS_\cM$ is triangulated with triangulated structure induced by the triangulated structure on $\cT$.
\end{theorem}

\begin{definition}[{\cite[Definition~1.2]{Neeman--2018--DetermineEachOther}}]
        Two metrics $\cM = \{B_t\}_{t \geq 0}$ and $\cM' = \{B'_t\}_{t \geq 0}$ on $\cT$ are {\em equivalent} if for every integer $t \geq 0$ there exists an integer $t' \geq 0$ such that $B'_{t'} \subseteq B_{t}$ and for every integer $s' \geq 0$ there exists an integer $s \geq 0$ such that $B_{s} \subseteq B'_{s'}$.
\end{definition}

The definitions of a Cauchy sequence and a compactly supported sequence with respect to a good metric $\cM$ only rely on the equivalence class of $\cM$. In particular, the completion of $\cT$ with respect to $\cM$ only depends on the equivalence class of $\cM$.

We introduce two new properties which mimic the properties of being a Cauchy sequence and a compactly supported sequence respectively, and which will be useful in concrete computations.

\begin{definition}
    Let $\mathbf{E} = (E_n,f_{n})_{n>0}$ be a sequence in $\cT$, and let $\cS$ be a subcategory of $\cT$.
    \begin{enumerate}
        \item The sequence $\mathbf{E}$ {\em stabilises at $\cS$} if, for all integers $m'>m \gg 0$, the cone of $f_{m,m'}$ lies in $\cS$.
        \item The sequence $\mathbf{E}$ is {\em compactly supported at $\cS$} if $\mocolim\mathbf{E}(\cS)=0$.
    \end{enumerate}
\end{definition}

\begin{remark}
    Let $\cM = \{B_t\}_{t\geq0}$ be a good metric on $\cT$. Then a sequence $\mathbf{E}$ is a Cauchy sequence if and only if for all $t\geq0$ it stabilises at $B_{t}$. Furthermore, a sequence $\mathbf{E}$ is compactly supported at $t\geq0$ if and only if $\mathbf{E}$ is compactly supported at $B_t$.
\end{remark}

\subsection{Subsequences}

For a sequence $\mathbf{E} = (E_n,f_n)_{n>0}$ in $\cT$ and a strictly increasing sequence $\mathbf{I}$ in $\bN$ we denote by $\mathbf{E}_\mathbf{I}$ the subsequence 
$(E_{m},f_{m,m'})_{m<m'\in \mathbf{I}}$ of $\mathbf{E}$ indexed by $\mathbf{I}$. 

\begin{remark} \label{R:subseq_are_cchy_and_cptly_supp}
	Let $\cS$ be a subcategory of $\cT$ and let $\mathbf{E}$ be a sequence in $\cT$. Then for each strictly increasing sequence $\mathbf{I}$ in $\bN$ the following hold:
    \begin{enumerate}
        \item $\mocolim \mathbf{E} \cong \mocolim \mathbf{E_I}$;
        \item if $\mathbf{E}$ stabilises at $\cS$, then so does $\mathbf{E_I}$;
        \item if $\mathbf{E}$ is compactly supported at $\cS$, then so is $\mathbf{E_I}$.
    \end{enumerate}
\end{remark}

\begin{remark} \label{R:subseq_induced_by_morphisms_being_zero_in_colimit}
    Let $\mathbf{E}=(E_n,f_n)_{n>0}$ be a sequence in $\cT$ and let $Z$ be an object in $\cT$ such that $\mocolim\mathbf{E}(Z) = 0$. Fix an integer $m>0$. Then for each morphism $\varphi \colon Z \rightarrow E_m$, there exists an integer $m_\varphi > m$ such that the composition
    \begin{center}
        \begin{tikzcd}
            Z \arrow[r, "\varphi"] & E_m \arrow[r, "f_{m,m_\varphi}"] & E_{m_\varphi}
        \end{tikzcd}
    \end{center}
    is zero. In particular, we have $\mocolim \mathbf{E} = 0$ if and only if $\mathbf{E}$ is a null sequence: Clearly, if $\mathbf{E}$ is a null sequence then its module colimit vanishes. Conversely, if $\mocolim \mathbf{E} = 0$ then in particular for all $m>0$ we have $\mocolim \mathbf{E}(E_m) = 0$, and hence the identity map in $\Hom_\cT(E_m,E_m)$ does not survive in the colimit $\colim \Hom_\cT(E_m,\mathbf{E}) = \mocolim \mathbf{E}(E_m) = 0$. Therefore, $\mathbf{E}$ must be a null-sequence.
    
    Moreover, if the group $\Hom_\cT(Z,E_m)$ is finitely generated as an abelian group, then there exists an integer $m_Z > m$ such that for all morphisms $\varphi' \colon Z \rightarrow E_m$ the composition
    \begin{center}
        \begin{tikzcd}
            Z \arrow[r, "\varphi'"] & E_m \arrow[r, "f_{m,m_Z}"] & E_{m_Z}
        \end{tikzcd}
    \end{center}
    is zero. If $\Hom_\cT(Z,E_m)$ is not finitely generated, then such an $m_Z$ may not exist. For example consider in $\cD(\Mod\bZ)$ the sequence
    \begin{center}
        \begin{tikzcd}
            \mathbf{E} = \bZ[\frac{1}{p}]/\bZ \arrow[r, "- \cdot p"] & \bZ[\frac{1}{p}]/\bZ \arrow[r, "- \cdot p"] & \bZ[\frac{1}{p}]/\bZ \arrow[r, "- \cdot p"] & \bZ[\frac{1}{p}]/\bZ \arrow[r, "- \cdot p"] & \cdots.
        \end{tikzcd}
    \end{center}
     We have
    \begin{equation*}
        \mocolim\mathbf{E}(\bZ) = \colim \Hom_{\cD(\Mod\bZ)}(\bZ,E) \cong \colim\mathbf{E} = 0,
    \end{equation*}
    but for all $a \in \bN$ the morphism $- \cdot \frac{1}{p^{a+1}} \colon \bZ \rightarrow \bZ[\frac{1}{p}]/\bZ$ is non-trivial and precomposes non-trivially with $f_{m,m+a} = - \cdot p^a$, and thus, no such $m_\bZ$ exists.
\end{remark}

\begin{definition}
    Let $\mathbf{E}=(E_n,f_n)_{n>0}$ be a sequence in $\cT$. For each integer $m>0$ let $\tilde{E}_m$ be a summand of $E_m$ with canonical inclusion and projection morphisms $\iota_m \colon \tilde{E}_m \rightarrow E_m$ and $\pi_m \colon E_m \rightarrow \tilde{E}_m$ respectively. We call the sequence $\mathbf{\tilde{E}} = (\tilde{E}_n,\pi_{n+1} \circ f_n \circ \iota_{n})_{n>0}$ a {\em component of $\mathbf{E}$}. 
\end{definition}

\begin{definition}
    Let $\cS$ be a subcategory of $\cT$. A sequence $\mathbf{E} = (E_n,f_n)_{n>0}$ is {\em $\cS$-trivial} if for all integers $m>0$ the entry $E_m$ has no non-trivial direct summands that lie in $\cS$. 
\end{definition}

\begin{lemma} \label{L:cptly_supp_means_no_summands}
   Assume $\cT$ is Krull-Schmidt and let $\cS$ and $\cR$ be subcategories of $\cT$ such that $\cR$ is closed under extensions and direct summands. Let $\mathbf{E}$ be a sequence in $\cT$ that is compactly supported at $\cS$. Then there exists a subsequence $\mathbf{E_I}$ of $\mathbf{E}$ and an $\cS$-trivial component $\mathbf{\tilde{E}_I}$ of $\mathbf{E_I}$ such that
    \[
        \mocolim \mathbf{E} \cong \mocolim \mathbf{\tilde{E}_I},
    \]
    and $\mathbf{\tilde{E}_I}$ is compactly supported at $\cS$. Moreover, if $\mathbf{E}$ stabilises at $\cR$, then so does $\mathbf{\tilde{E}_I}$.
\end{lemma}

\begin{proof}
    Let $\mathbf{E} = (E_n,f_n)_{n>0}$. For each $m > 0$, let $\tilde{E}_m^c$ be the largest summand of $E_m$ that lies in $\cS$, and let $E_m \cong \tilde{E}_m^c \oplus \tilde{E}_m$ be the corresponding decomposition of $E_m$. For each $m>0$, let $\pi_m \colon E_m \rightarrow \tilde{E}_m$ be the canonical projection morphism and $\kappa_m \colon \tilde{E}^c_m \rightarrow E_m$ be the canonical inclusion morphism. Then, as $\mathbf{E}$ is compactly supported at $\cS$, by Remark~\ref{R:subseq_induced_by_morphisms_being_zero_in_colimit}, there exists a subsequence $\mathbf{E}_\mathbf{I}$ of $\mathbf{E}$ such that for all $m,m' \in \mathbf{I}$ with $m<m'$ we have $f_{m,m'}\circ \kappa_m (\tilde{E}_m^c)=0$. Consider the components $\mathbf{\tilde{E}_I} = (\tilde{E}_n,\tilde{f}_n)_{n \in \mathbf{I}}$ and $\mathbf{\tilde{E}^c_I} = (\tilde{E}_n^c,\tilde{f}^c_n)_{n \in \mathbf{I}}$ of $\mathbf{E}_\mathbf{I}$. Then for all $m,m' \in \mathbf{I}$ with $m<m'$ since $f_{m,m'} \circ \kappa_m(\tilde{E}_m^c) = 0$ we have $\tilde{f}^c_{m,m'} = 0$.
    Consequently, for all $m \in \mathbf{I}$ with successor $m' \in \mathbf{I}$ we obtain a morphism of split triangles
	\begin{center}
		\begin{tikzcd}
			\tilde{E}_m^c \arrow[r,"\kappa_{m}"] \arrow[d,"0"] & E_{m} \arrow[d, "f_{m,m'}"] \arrow[r,"\pi_{m}"] & \tilde{E}_{m} \arrow[r, "0"] \arrow[d,"\tilde{f}_{m,m'}"] & \Sigma \tilde{E}_m^c \arrow[d,"0"]
			\\
			\tilde{E}_{m'}^c \arrow[r,"\kappa_{m'}"] & E_{m'} \arrow[r,"\pi_{m'}"] & \tilde{E}_{m'} \arrow[r, "0"] & \Sigma \tilde{E}_{m'}^c.
		\end{tikzcd}
	\end{center}
    As $\mathbf{\tilde{E}^c}_\mathbf{I}$ is a null sequence and $\yo$ is homological we obtain \[
        \mocolim \mathbf{E}\cong \mocolim \mathbf{E}_\mathbf{I} \cong \mocolim \mathbf{\tilde{E}}_\mathbf{I}.
    \]
    Thus, as $\mathbf{E}$ is compactly supported at $\cS$, so is $\mathbf{\tilde{E}_I}$.
    Assume now that $\mathbf{E}$ stabilises at $\cR$. Then for all integers $m'>m \gg 0$ we have $\cone f_{m,m'} \in \cR$. The $3 \times 3$ Lemma (cf.\ for example \cite[Lemma~2.6]{May-triangulated}) yields a diagram
    \begin{center}
		\begin{tikzcd}
			\tilde{E}_m^c \arrow[r,"\kappa_{m}"] \arrow[d,"0"] & E_{m} \arrow[d, "f_{m,m'}"] \arrow[r,"\pi_{m}"] & \tilde{E}_{m} \arrow[r, "0"] \arrow[d,"\tilde{f}_{m,m'}"] & \Sigma \tilde{E}_m^c \arrow[d,"0"]
			\\
			\tilde{E}_{m'}^c \arrow[r,"\kappa_{m'}"] \arrow[d] & E_{m'} \arrow[r,"\pi_{m'}"] \arrow[d] & \tilde{E}_{m'} \arrow[r, "0"] \arrow[d] & \Sigma \tilde{E}_{m'}^c \arrow[d]
            \\
            \tilde{E}_{m'}^c \oplus \Sigma \tilde{E}_m^c \arrow[r] \arrow[d] & \cone{f_{m,m'}} \arrow[r] \arrow[d] & \cone \tilde{f}_{m,m'} \arrow[r] \arrow[d] & \Sigma \tilde{E}_{m'}^c \oplus \Sigma^2 \tilde{E}_m^c \arrow[d]
            \\
            \Sigma \tilde{E}_m^c \arrow[r,"\Sigma \kappa_{m}"] & \Sigma E_{m} \arrow[r,"\Sigma \pi_{m}"] & \Sigma \tilde{E}_{m} \arrow[r, "0"] & \Sigma^2 \tilde{E}_m^c,
		\end{tikzcd}
	\end{center}
    which is commutative everywhere except for the bottom right square which is anticommutative, and whose rows and columns are distinguished triangles in $\cT$. By the bottom right square and the third row there exists an object $D_{m,m'} \in \cT$ such that $\cone f_{m,m'} \cong \Sigma \tilde{E}_m^c \oplus D_{m,m'}$ and there exists a triangle
    \begin{center}
		\begin{tikzcd}
            \tilde{E}_{m'}^c \arrow[r] & D_{m,m'} \arrow[r] & \cone \tilde{f}_{m,m'} \arrow[r] & \Sigma \tilde{E}_{m'}^c.
		\end{tikzcd}
	\end{center}
    Therefore, as $\cone{f_{m,m'}} \in \cR$ and $\cR$ is closed under extensions and direct summands, we also have $\cone{\tilde{f}_{m,m'}} \in \cR$. So, $\mathbf{\tilde{E}_I}$ stabilises at $\cR$.
\end{proof}

\section{Metric completions and t-structures} \label{S:metric completions} 
In this section we again fix a triangulated category $\cT$ with shift functor $\Sigma \colon \cT \rightarrow \cT$.

\subsection{Metrics from t-structures}

T-structures provide an important source for metrics of $\cT$.

\begin{definition}[\cite{BBD}]
    A t-structure on $\cT$ is a pair of full subcategories $(\aisle,\coaisle)$ of $\cT$, both closed under direct summands, such that
    \begin{enumerate}
        \item{for all $X \in \aisle$ and $Y \in \coaisle$ we have $\Hom_\cT(X,Y) = 0$;}
        \item{for all $T \in \cT$ there exists a distinguished triangle
        \begin{center}
		\begin{tikzcd}
            X \arrow[r,"\alpha"] & T \arrow[r,"\beta"] & Y \arrow[r,"\gamma"] & \Sigma X,
		\end{tikzcd}
	\end{center}
        with $X \in \aisle$ and $Y \in \coaisle$, called an {\em approximation triangle of $T$ with respect to $(\aisle,\coaisle)$};}
        \item there is an inclusion $\Sigma \aisle \subseteq \aisle$.
    \end{enumerate}
    Given a t-structure $(\aisle, \coaisle)$ the subcategory $\aisle$ is called its {\em aisle} and $\cY$ its {\em coaisle}.
\end{definition}

The approximation triangle of an object $T \in \cT$ with respect to a fixed t-structure $(\aisle, \coaisle)$ is functorial and unique up to isomorphism of triangles. More specifically, we have essentially unique truncation functors $\tau_\aisle \colon \cT \to \aisle$ and $\tau^\coaisle \colon \cT \to \coaisle$ so that for any morphism $f \colon T \to T'$ there is a unique morphism $\Delta f$ of approximation triangles
\begin{center}
\begin{equation}\label{E:map_of_approximation_triangles}
    \begin{tikzcd}
 	\Delta T \ar[d,"\Delta f"]& = &	\tau_\aisle T \arrow[r,"\alpha"] \ar[d,"\tau_\aisle f"] & T \arrow[r,"\beta"] \ar[d,"f"]  & \tau^\coaisle T	\arrow[r,"\gamma"] \ar[d,"\tau^\coaisle f"]  & \tau_\aisle \Sigma T \ar[d]
           \\
            \Delta T' & = &         \tau_\aisle T' \arrow[r,"\alpha'"] & T'\arrow[r,"\beta'"] & \tau^\coaisle T' \arrow[r,"\gamma'"] & \tau_\aisle \Sigma T',
		\end{tikzcd}
  \end{equation}
 \end{center}
where the top row $\Delta T$ is the approximation triangle of $T$, and the bottom row $\Delta T'$ is the approximation triangle of $T'$.

\begin{definition}
    Let $(\aisle, \coaisle)$ be a t-structure on $\cT$. The {\em aisle metric $\cM_\aisle$} is defined to be the good metric $\cM_\aisle = \{(\cM_\aisle)_t\}_{t \geq 0}$ such that $(\cM_\aisle)_0 = \cT$ and for all integers $t >0$ we have
    \[
        (\cM_\aisle)_t = \Sigma^t \aisle.
    \]
    Symmetrically, the {\em coaisle metric $\cM_\coaisle$} is defined to be the good metric $\cM_\coaisle = \{(\cM_\coaisle)_t\}_{t \geq 0}$ such that $(\cM_\coaisle)_0 = \cT$ and for all integers $t >0$ we have
    \[
        (\cM_\coaisle)_t = \Sigma^{-t} \coaisle.
    \]
\end{definition}

\begin{definition}[{\cite[Definition 0.10]{neeman2018Brown}}]
    Two t-structures $(\aisle,\coaisle)$ and $(\aisle',\coaisle')$ on $\cT$ are called {\em equivalent}, if there exists an integer $t \geq 0$ such that $\Sigma^t \aisle \subseteq \aisle' \subseteq \Sigma^{-t} \aisle$.
\end{definition}

The notion of equivalent t-structures defines an equivalence relation on the t-structures in $\cT$. It follows from the definition that equivalent t-structures give rise to equivalent metrics.

\subsection{Approximations of sequences}

Under the presence of a t-structure $(\aisle, \coaisle)$, any sequence $\mathbf{E}$ in $\cT$ gives rise to two approximating sequences, one in $\aisle$ and one in $\coaisle$ which inherit certain properties of $\mathbf{E}$. Throughout the rest of Section~\ref{S:metric completions}, we fix a triangulated category $\cT$ with a t-structure $(\aisle,\coaisle)$ and truncation functors $\tau_\aisle$ and $\tau^\coaisle$.

\begin{definition}
    Let $\mathbf{E} = (E_n,f_n)_{n>0}$ be a sequence in $\cT$. We call the sequence 
    \[
        \tau_\aisle \mathbf{E} = (\tau_\aisle E_n,\tau_\aisle f_n)_{n>0}
    \]
    the {\em $\aisle$-approximation of $\mathbf{E}$} and the sequence  
    \[
        \tau^\coaisle \mathbf{E} = (\tau^\coaisle E_n,\tau^\coaisle f_n)_{n>0}
    \]
    the {\em $\coaisle$-approximation of $\mathbf{E}$.} We call the sequence $\Delta \mathbf{E} = (\Delta E_n, \Delta f_n)_{n>0}$, where $\Delta E_n$ is the approximation triangle of $E_n$ and $\Delta f_n$ is the map of triangles from (\ref{E:map_of_approximation_triangles}), the {\em sequence of approximation triangles of $\mathbf{E}$.}
\end{definition}

A triangulated category $\cT$ is {\em $d$-Calabi-Yau} for an integer $d > 0$ if it is $\bK$-linear and $\Sigma^d$ is a Serre functor of $\cT$.

\begin{lemma} \label{L:approximations_S_trivial}
    Let $\mathbf{E} = (E_n,f_n)_{n>0}$ be a sequence in $\cT$ with $\aisle$-approximation $\mathbf{X} = (X_n,h_n)_{n>0}$ and $\coaisle$-approximation $\mathbf{Y} = (Y_n,g_n)_{n>0}$. 
    If $\cT$ is $d$-Calabi-Yau for some integer $d > 0$ then the following hold:
    \begin{enumerate}
        \item 
        If $\mathbf{E}$ is $(\Sigma^{d-1}\aisle)$-trivial, then so is $\mathbf{X}$.
        \item 
        If $\mathbf{E}$ is $(\Sigma^{-(d-1)}\coaisle)$-trivial, then so is $\mathbf{Y}$.
    \end{enumerate}
\end{lemma}

\begin{proof}
    Fix an integer $m>0$, and let
    \begin{tikzcd}
 		X_m \arrow[r,"\alpha"]  & E_m \arrow[r,"\beta"]   & Y_m	\arrow[r,"\gamma"]   & \Sigma X_m
		\end{tikzcd}
    be the approximation triangle of $E_m$. Then $\alpha$ is a $\aisle$-cover and $\beta$ is a $\coaisle$-envelope. Let $T \in \cT$ be a summand of $X_m$ that lies in $\Sigma^{d-1}\aisle$ with canonical projection $\pi_T \colon X_m \rightarrow T$. Since $\cT$ is $d$-Calabi-Yau, we have 
    \begin{equation*}
        \Hom_\cT(\Sigma^{-1}Y_m,T) \cong D\Hom_\cT(T, \Sigma^{d-1}Y_m) =0.
    \end{equation*}
    Thus the morphism $\pi_T$ factors through $\alpha$, and so, $T$ is also a summand of $E_m$. This shows (i), and (ii) follows dually.
\end{proof}

\begin{lemma} \label{L:cptly_supp_prop_of_X_and_Y}
    Let $\cS$ be a subcategory of $\cT$. Let $\mathbf{E}$ be a sequence in $\cT$ that is compactly supported at $\cS$ with $\aisle$-approximation $\mathbf{X}$ and $\coaisle$-approximation $\mathbf{Y}$. Then the following hold:
    \begin{enumerate}
        \item The sequence $\mathbf{X}$ is compactly supported at $\Sigma^{-1}\aisle \cap \cS$. Moreover, if $\aisle \subseteq \cS$, then $\mathbf{X}$ is a null sequence and there is an isomorphism
        \begin{equation*}
            \mocolim\mathbf{E} \cong \mocolim\mathbf{Y}.
        \end{equation*}
        \item If $\cT$ is $d$-Calabi-Yau for some $d>0$, then the sequence $\mathbf{Y}$ is compactly supported at $\cS \cap \Sigma^{-(d-1)}\coaisle$.
    \end{enumerate}
\end{lemma}

\begin{proof}
Since $\yo$ is homological, for all objects $T \in \cT$, we have an exact sequence
    \begin{equation}\label{E:les}
    	\mocolim\Sigma^{-1}\mathbf{Y}(T) \to \mocolim\mathbf{X} (T) \to \mocolim\mathbf{E}(T) \to \mocolim \mathbf{Y}(T) \to \mocolim \mathbf{\Sigma X}(T).
    \end{equation}
    \begin{enumerate}
        \item Suppose that $T$ is an object in $\Sigma^{-1}\aisle \cap \cS$. Since $(\aisle,\coaisle)$ is a t-structure and $T \in \Sigma^{-1}\aisle$, we have $\mocolim \Sigma^{-1}\mathbf{Y}(T) = 0$. Also, since $\mathbf{E}$ is compactly supported at $\cS$ and $T \in \cS$, we have $\mocolim \mathbf{E}(T) = 0$. Hence the claim follows by the exactness of (\ref{E:les}).

        Now suppose that $\aisle$ is a subcategory of $\cS$. Then $\aisle$ is also a subcategory of $\Sigma^{-1}\aisle \cap \cS$ and for all integers $m>0$ we have $\mocolim\mathbf{X}(X_m)=0$. Thus $\mathbf{X}$ is a null sequence, and the isomorphism $\mocolim \mathbf{E} \cong \mocolim \mathbf{Y}$ is induced by the exact sequence (\ref{E:les}).

        \item Suppose that $\cT$ is $d$-Calabi-Yau. Then, for each integer $m>0$, we have
        \begin{equation*}
            \Hom_\cT(T,\Sigma X_m) \cong D\Hom_\cT(\Sigma X_m, \Sigma^{d} T) \cong D\Hom_\cT(\Sigma^{-(d-1)}X_m, T).
        \end{equation*}
        Thus the claim follows similarly to (i) by the exactness of (\ref{E:les}).
    \end{enumerate}
\end{proof}

\begin{lemma} \label{L:cauchy_prop_of_X_and_Y}
    Let $\cR$ be an extension closed subcategory of $\cT$. Let $\mathbf{E} = (E_n,f_n)_{n>0}$ be a sequence in $\cT$ that stabilises at $\cR$ with $\aisle$-approximation $\mathbf{X} = (X_n,h_n)_{n>0}$ and $\coaisle$-approximation $\mathbf{Y} = (Y_n,g_n)_{n>0}$. Then the following hold:
    \begin{enumerate}
        \item The sequence $\mathbf{X}$ stabilises at $(\cR \cap \aisle)$ if at least one of the following holds:
        \begin{enumerate}
            \item there is an inclusion $\coaisle \subseteq \cR$;
            \item there is an inclusion $\cR \subseteq \Sigma\aisle$. Moreover, in this case, the sequence $\mathbf{Y}$ stabilises at $0$, and, for all integers $m'>m\gg0$, there is an isomorphism $\cone{f_{m,m'}} \cong \cone{h_{m,m'}}$.
        \end{enumerate}
        \item The sequence $\mathbf{Y}$ stabilises at $(\cR \cap \Sigma\coaisle)$ if at least one of the following holds:
        \begin{enumerate}
            \item there is an inclusion $\Sigma\aisle \subseteq \cR$;
            \item there is an inclusion $\cR \subseteq \coaisle$. Moreover, in this case, the sequence $\mathbf{X}$ stabilises at $0$ and, for all integers $m'>m\gg0$, there is an isomorphism $\cone{f_{m,m'}} \cong \cone{g_{m,m'}}$.
        \end{enumerate}
    \end{enumerate}
\end{lemma}

\begin{proof}
    Fix integers $m'>m \gg 0$ so that $\cone(f_{m,m'}) \in \cR$. Note that as $h_{m,m'} \colon X_m \rightarrow X_{m'}$ is a morphism between two objects in $\aisle$, and aisles are closed under extensions and positive shifts, its cone also lies in $\aisle$. Similarly, as $g_{m,m'} \colon Y_m \rightarrow Y_{m'}$ is a morphism between two objects in $\coaisle$, and coaisles are closed under extensions and negative shifts, its cone lies in $\Sigma \coaisle$. Moreover, the sequence of approximation triangles $\Delta \mathbf{E}$ induces a diagram 
    \begin{center}
    \begin{equation}\label{E:3x3again}
		\begin{tikzcd}
			X_m \arrow[r] \arrow[d,"h_{m,m'}"] & E_{m} \arrow[d, "f_{m,m'}"] \arrow[r] & Y_m \arrow[r] \arrow[d,"g_{m,m'}"] & \Sigma X_m \arrow[d,"\Sigma h_{m,m'}"]
			\\
			X_{m'} \arrow[r] \arrow[d] & E_{m'} \arrow[d] \arrow[r] & Y_{m'} \arrow[r] \arrow[d] & \Sigma X_{m'} \arrow[d]
            \\
            \cone(h_{m,m'}) \arrow[r] \arrow[d] & \cone(f_{m,m'}) \arrow[r] \arrow[d] & \cone (g_{m,m'})\arrow[r] \arrow[d] & \Sigma \cone(h_{m,m'}) \arrow[d]
            \\
            \Sigma X_m \arrow[r] & \Sigma E_{m} \arrow[r] & \Sigma Y_m \arrow[r] & \Sigma^2 X_m
		\end{tikzcd}  
    \end{equation}
	\end{center}
    that is commutative everywhere except in the bottom right square, which is anticommutative, and whose rows and columns are distinguished triangles in $\cT$.
    \begin{enumerate}
        \item \begin{enumerate}
            \item If $\coaisle$ is a subcategory of $\cR$, then the statement follows by the triangle in the third row of (\ref{E:3x3again}) and the extension closure of $\cR$.
            \item  Suppose that $\cR$ is a subcategory of $\Sigma\aisle$. Then, by the third row of (\ref{E:3x3again}) and the extension closure of $\Sigma\aisle$, the cone of $g_{m,m'}$ lies in $\Sigma\coaisle \cap \Sigma\aisle =\{0\}$. So, the cones of $f_{m,m'}$ and $h_{m,m'}$ are isomorphic and lie in $\cR \cap \aisle$. In particular, $\mathbf{X}$ stabilises at $(\cR \cap \aisle)$ and $\mathbf{Y}$ stabilises at $0$.
        \end{enumerate}

        \item Follows similarly to (i).
    \end{enumerate}
\end{proof}

\begin{remark}
Note that Lemma~\ref{L:cauchy_prop_of_X_and_Y}(i)(b) and (ii)(b) have strong implications for the sequence $\mathbf{Y}$ and $\mathbf{X}$ respectively: A sequence stabilises at $0$ if and only if its morphisms are eventually all isomorphisms. Or, equivalently, all the sequence entries are eventually isomorphic to an object $T \in \cT$ and the module colimit of the sequence is isomorphic to $\yo(T)$.
\end{remark}

\subsection{Aisle metric completions}
 
Approximations of sequences provide a useful tool to compute completions with respect to good metrics which contain an aisle metric.

\begin{definition}
    Fix metrics $\cM = \{B_t\}_{t \geq 0}$ and $\cM' = \{B'_t\}_{t \geq 0}$ on $\cT$. We write
    \[
        \cM \subseteq \cM'
    \]
    if for all integers $t \geq 0$ we have $B_t \subseteq B'_t$.
\end{definition}

\begin{proposition} \label{P:aisle_inside_metric}
    Let $\cT$ be a triangulated category equipped with a good metric $\cM = \{B_t\}_{t\geq0}$ and a t-structure $(\aisle,\coaisle)$ such that $\cM_\aisle \subseteq \cM$. Let $\mathbf{E}$ be a sequence in $\cT$ that is compactly supported at $s\geq0$ and stabilises at $B_r$ for some $r \geq 0$. Let $\mathbf{Y}$ be the $\Sigma^s \coaisle$-approximation of $\mathbf{E}$. In $\Mod\cT$ there is an isomorphism
        \begin{equation*}
            \mocolim\mathbf{E} \cong \mocolim\mathbf{Y};
        \end{equation*}
    and the following hold:
    \begin{enumerate}
        \item the sequence $\mathbf{Y}$ is compactly supported at $s$;
        \item if $r \leq s+1$, then the sequence $\mathbf{Y}$ stabilises at $(B_{r} \cap \Sigma^{s+1} \coaisle)$.
    \end{enumerate}
    In particular, if we have equalities $\cM = \cM_\aisle$ and $r=s+1$, then $\mathbf{Y}$ stabilises at $0$.
\end{proposition}

\begin{proof}
    Apply Lemma~\ref{L:cptly_supp_prop_of_X_and_Y}(i) with  $\cS=B_s$ to obtain an isomorphism
    $\mocolim\mathbf{E} \cong \mocolim\mathbf{Y}$. Thus, $\mathbf{Y}$ is also compactly supported at $s$. Apply Lemma~\ref{L:cauchy_prop_of_X_and_Y}(ii)(a) with $\cR=B_{r}$ to show that $\mathbf{Y}$ stabilises at $B_r \cap \Sigma^{s+1} \coaisle$. Therefore, if $\cM=\cM_\aisle$ and $r=s+1$, then $\mathbf{Y}$ stabilises at $\Sigma^{r} \aisle \cap \Sigma^{s+1} \coaisle = \{0\}$.
\end{proof}

\begin{corollary} \label{C:aisle_in_metric_completion}
    Let $(\aisle,\coaisle)$ be a t-structure on $\cT$ and consider the aisle metric $\cM_\aisle$. Then we have
    \begin{equation*}
        \fS_{\cM_\aisle} \simeq \add \bigcup_{p\in\bZ} \Sigma^p\coaisle.
    \end{equation*}
In particular, the completion $\fS_{\cM_\aisle}$ is equivalent to a thick subcategory of $\cT$.
\end{corollary}

\begin{proof}
    Fix $E \in \fS_{\cM_\aisle}$. Then there exists a sequence $\mathbf{E}$ in $\cT$ that is Cauchy and compactly supported at $t\geq0$ with respect to $\cM_\aisle$ such that $E \cong \mocolim \mathbf{E}$. By Proposition~\ref{P:aisle_inside_metric} with $s=t$ and $r=t+1$, we have that the $\Sigma^t\coaisle$-approximation $\mathbf{Y}$ of $\mathbf{E}$ stabilises at $0$, and so, there exists an object $Y \in \Sigma^t\coaisle$ such that $\yo(Y) \cong \mocolim\mathbf{Y} \cong \mocolim \mathbf{E}$. Conversely, any object in $\add \bigcup_{p\in\bZ} \yo(\Sigma^p\coaisle) \simeq \add \bigcup_{p\in\bZ} \Sigma^p\coaisle$ can be obtained as a colimit of a constant sequence in $\yo(\Sigma^p \coaisle)$ for some $p \in \bZ$, and hence as a module colimit of a compactly supported Cauchy sequence in $\cT$.
\end{proof}

\begin{remark}
    The result from Corollary \ref{C:aisle_in_metric_completion} has since been obtained independently in \cite[Corollary~3.6]{BCMPZ-completion} using their previously developed ideas on extendable t-structures.
\end{remark}

\subsection{Coaisle metric completions}

We now focus on completions with respect to coaisle metrics. Like in the case of aisle metrics, approximations can be used to greatly reduce the sequences whose module colimits we have to compute. This reduction is particularly ruthless in $d$-Calabi-Yau triangulated categories.

\begin{proposition} \label{P:coaisle_inside_metric}
    Let $\cT$ be a triangulated category equipped with a good metric $\cM = \{B_t\}_{t\geq0}$. Let $\mathbf{E}$ be a sequence in $\cT$ that is compactly supported at $s\geq0$ and stabilises at $B_r$ for some $r \geq 0$. Let $(\aisle,\coaisle)$ be a t-structure on $\cT$ and let $\mathbf{X}$ and $\mathbf{Y}$ denote the $\aisle$-approximation and $\coaisle$-approximation of $\mathbf{E}$ respectively. Then the following hold:
    \begin{enumerate}
        \item $\mathbf{X}$ is compactly supported at $B_s \cap \Sigma^{-1}\aisle$;
        \item if $\cT$ is $d$-Calabi-Yau for some $d>0$, then $\mathbf{Y}$ is compactly supported at $B_s \cap \Sigma^{-(d-1)}\coaisle$.
    \end{enumerate}
    If $\cM \subseteq \cM_{\coaisle}$, 
    then 
    \begin{enumerate}[resume]
        \item $\mathbf{X}$ stabilises at $0$;
        \item $\mathbf{Y}$ stabilises at $(B_{r} \cap \Sigma \coaisle)$; 
        \item {$\mathbf{Y}$ is Cauchy if and only if $\mathbf{E}$ is Cauchy.}
    \end{enumerate}
\end{proposition}

\begin{proof}
    Apply Lemma~\ref{L:cptly_supp_prop_of_X_and_Y}(i) with $\cS=B_s$ to obtain that $\mathbf{X}$ is compactly supported at $B_s \cap \Sigma^{-1}\aisle$ and Lemma~\ref{L:cptly_supp_prop_of_X_and_Y}(ii) to obtain that if $\cT$ is $d$-Calabi-Yau then $\mathbf{Y}$ is compactly supported at $B_s \cap \Sigma^{-(d-1)}\coaisle$. If $\cM \subseteq \cM_\coaisle$, then $B_r \subseteq \Sigma^{-r}\coaisle \subseteq \coaisle$. Thus, by Lemma~\ref{L:cauchy_prop_of_X_and_Y}(ii)(b) with $\cR = B_r$ we have that $\mathbf{X}$ stabilises at $0$, and $\mathbf{Y}$ stabilises at $B_r \cap \Sigma \coaisle$ and is Cauchy if and only if $\mathbf{E}$ is Cauchy.
\end{proof}

\begin{definition} \label{D:coaisle_sequence}
    Let $(\aisle,\coaisle)$ be a t-structure on $\cT$. A {\em coaisle sequence with respect to the coaisle $\coaisle$} is a sequence $\mathbf{Y} = (Y_n,g_n)_{n>0}$ in $\coaisle$ that is compactly supported at $\Sigma^{-1}\coaisle$ and stabilises at $\Sigma^{-2}\coaisle$. A {\em minimal coaisle sequence with respect to the coaisle $\coaisle$} is a $(\Sigma^{-1}\coaisle)$-trivial coaisle sequence with respect to $\coaisle$.
\end{definition}

Coaisle sequences naturally occur as approximations of compactly supported Cauchy sequences with respect to coaisle metrics.

\begin{corollary} \label{C:coaisle_metric}
    Let $\cT$ be a 2-Calabi-Yau triangulated category equipped with a t-structure $(\aisle,\coaisle)$ and the coaisle metric $\cM_\coaisle$. Let $\mathbf{E}$ be a sequence in $\cT$ that is compactly supported at $t\geq0$ and stabilises at $B_{t+1}$. Let $\mathbf{X}$ and $\mathbf{Y}$ be the $(\Sigma^{-(t-1)}\aisle)$-approximation and $(\Sigma^{-(t-1)}\coaisle)$-approximation of $\mathbf{E}$ respectively. Then $\mathbf{X}$ stabilises at $0$ and $\mathbf{Y}$ is a coaisle sequence with respect to $\Sigma^{-(t-1)}\coaisle$.
\end{corollary}

\begin{proof}
    Apply Proposition~\ref{P:coaisle_inside_metric} with t-structure $(\Sigma^{-(t-1)}\aisle,\Sigma^{-(t-1)}\coaisle)$, $s=t$, $r=t+1$, and $d=2$.
\end{proof}

For future computations, it proves useful to combine Proposition~\ref{P:aisle_inside_metric} and Corollary~\ref{C:coaisle_metric}.

\begin{proposition} \label{P:combine_aisle_and_coaisle}
    Let $\cM =\{B_t\}_{t\geq0}$ be a good metric on a 2-Calabi-Yau category $\cT$. Let $(\widehat{\aisle},\widehat{\coaisle})$ be a t-structure on $\cT$ such that $\cM_{\widehat{\aisle}} \subseteq \cM$ and let $E \in \fS_\cM(\cT)$ be compactly supported at $t\geq0$. Then there exists a sequence $\mathbf{\widehat{Y}}$ in $\Sigma^t\widehat{\coaisle}$ such that 
        \[
            E \cong \mocolim \widehat{\mathbf{Y}}
        \]
    and such that, if $\coaisle = \Sigma^{t+1} B_{t+1} \cap \Sigma^{2(t+1)}\widehat{\coaisle}$ is a coaisle of $\cT$ with corresponding t-structure $(\aisle,\coaisle)$, the following hold:
        \begin{enumerate}
            \item the $(\Sigma^{-(t-1)}\aisle)$-approximation of $\mathbf{\widehat{Y}}$ stabilises at $0$;
            \item the $(\Sigma^{-(t-1)}\coaisle)$-approximation of  $\mathbf{\widehat{Y}}$ is a minimal coaisle sequence with respect to $\Sigma^{-(t-1)}\coaisle$.
        \end{enumerate}
\end{proposition}

\begin{proof}
    As $E \in \fS_\cM(\cT)$ there exists a sequence $\mathbf{E}$ in $\cT$ that is Cauchy with respect to $\cM$ such that $E \cong \mocolim\mathbf{E}$. Since $E$ is compactly supported at $t$, so is $\mathbf{E}$. Thus, as $\cM_{\widehat{\aisle}} \subseteq \cM$, by Proposition~\ref{P:aisle_inside_metric} with $s=t$ and $r=t+1$, there exists a sequence $\mathbf{\widehat{Y}}$ in $\Sigma^t\widehat{\coaisle}$ that is compactly supported at $t$ and stabilises at $B_{t+1} \cap \Sigma^{t+1}\widehat{\coaisle}$ such that $E \cong \mocolim\mathbf{\widehat{Y}}$. By Lemma~\ref{L:cptly_supp_means_no_summands} we may assume without loss of generality that $\mathbf{\widehat{Y}}$ is $B_t$-trivial.

    Now let $\coaisle = \Sigma^{t+1} B_{t+1} \cap \Sigma^{2(t+1)}\widehat{\coaisle}$. The sequence $\mathbf{\widehat{Y}}$ is compactly supported at $B_t$, and, because $\cM$ is a good metric, we have 
    \begin{equation*}
        \Sigma^{-t}\coaisle = \Sigma B_{t+1} \cap \Sigma^{t+2}\widehat{\coaisle} \subseteq \Sigma B_{t+1} \subseteq B_t.
    \end{equation*}
    Therefore, $\widehat{\mathbf{Y}}$ is compactly supported at $\Sigma^{-t}\coaisle$ and stabilises at $\Sigma^{-(t+1)}\coaisle$. Thus, if $\coaisle$ is a coaisle of $\cT$, we may apply Corollary~\ref{C:coaisle_metric} with coaisle metric $\cM_\coaisle$ to obtain that the $(\Sigma^{-(t-1)}\aisle)$-approximation of $\mathbf{\widehat{Y}}$ stabilises at $0$ and the $(\Sigma^{-(t-1)}\coaisle)$-approximation of $\mathbf{\widehat{Y}}$ is a coaisle sequence with respect to $\Sigma^{-(t-1)}\coaisle$.
    Finally, as $\Sigma^{-t}\coaisle \subseteq B_t$, the sequence $\mathbf{\widehat{Y}}$ is also $\Sigma^{-t}\coaisle$-trivial, and by Lemma~\ref{L:approximations_S_trivial}(iv) so is its $(\Sigma^{-(t-1)} \coaisle)$-approximation.
    \end{proof}

\section{Discrete cluster categories of type \texorpdfstring{$A$}{A}}
Igusa and Todorov \cite{ITcyclic} introduced discrete cluster categories of infinite rank over a fixed field $\bK$. These categories are $\bK$-linear, Hom-finite, triangulated categories which exhibit interesting cluster combinatorics of infinite Dynkin type $A$, and generalise the category studied by Holm and J{\o}rgensen in \cite{HJ-cat}. They are constructed combinatorially from a discrete subset $\cZ \subseteq S^1$.

Consider a discrete (with respect to the standard topology) subset $\cZ \subseteq S^1$ of the unit circle satisfying the following:
\begin{enumerate}[label=\textbf{Z\arabic*}]
    \item{The set $\cZ$ is infinite.}
    \item{The set $\cZ$ has $N \in \bN$ (proper) accumulation points, i.e.\ limits of eventually non-constant sequences from $\cZ$.}
    \item{Each accumulation point $a$ of $\cZ$ is {\em two-sided}. That is, there exist sequences $\mathbf{z} = \{z_i\}_{i > 0}$ and $\mathbf{z'} = \{z'_i\}_{i > 0}$ from $\cZ$ such that $\mathbf{z}$ converges to $a$ from a clockwise direction and $\mathbf{z'}$ converges to $a$ from an anti-clockwise direction.}
\end{enumerate}

There exists a natural cyclic order on the set $\cZ$ induced by the cyclic order on $S^1$. Given points $z_1, z_2 \ldots, z_k \in S^1$ we write
\[
    z_1 < z_2 < \ldots < z_k
\]
if the points are pairwise distinct and, when walking along the circle in an anti-clockwise direction, we encounter first $z_1$, then $z_2$, etc., and last $z_k$. Analogously, we use weak inequalities $\leq$ when two points might coincide. Given points $a,c \in S^1$ we write $(a,c)$ for the interval
\begin{equation*}
    (a,c) = \begin{cases}
        \{b \in S^1 \mid a < b < c\} &\text{ if } a \neq c,
        \\
        \varnothing &\text{ if } a=c.
    \end{cases}
\end{equation*}
Analogously we define the closed interval $[a,c]$ and the half-open intervals $[a,c)$ and $(a,c]$ by using weak inequalities where appropriate. Note that these intervals lie in the circle $S^1$, not $\cZ$, and as such, may contain accumulation points.

Axioms \textbf{Z1}-\textbf{Z3} ensure that each point $z \in \cZ$ has a unique successor $z^+$, i.e.\ a unique point $z^+ \in \cZ$ such that $(z,z^+) \cap \cZ = \varnothing$. Symmetrically, it has a unique predecessor $z^-$, i.e.\ a unique point $z^- \in \cZ$ such that $(z^-,z) \cap \cZ = \varnothing$. For an integer $n > 0$ we iteratively define $z^{(0)} = z$, $z^{(n)} = (z^{(n-1)})^+$ and $z^{(-n)} = (z^{(-n+1)})^-$.

We denote by $\overline{\cZ}$ the topological closure of $\cZ$ with respect to the standard topology. We call the points in $L(\cZ) = \overline{\cZ} \setminus \cZ$ the accumulation points of $\cZ$. We label the accumulation points in $L(\cZ)$ consecutively, in an anti-clockwise direction, by $a_1, \ldots, a_N$, where $N=|L(\cZ)|$. We consider the indices of the accumulation points modulo $N$, with $a_{i+jN} = a_{i}$ for all $i \in [N] = \{1,2,\dots,N\}$ and for all $j \in \bZ$.

An {\em arc of $\cZ$} is a two-element subset $Z = \{z_0,z_1\}$ of $\cZ$ such that $z_1$ is neither the predecessor of $z_0$, nor the successor of $z_0$, i.e.\ $z_1 \notin \{z_0^+,z_0,z_0^-\}$. We call the points $z_0,z_1 \in \cZ$ the {\em endpoints of the arc $Z= \{z_0,z_1\}$}. We say that two arcs $X$ and $Y$ cross, if $X = \{x_0,x_1\}$ and $Y=\{y_0,y_1\}$ such that
\begin{equation*}
    x_0 < y_0 < x_1 < y_1.
\end{equation*}
Given an arc $X = \{x_0,x_1\}$ we write $X^-$ for the arc $\{x_0^-,x_1^-\}$ and $X^+$ for the arc $\{x_0^+,x_1^+\}$.

\begin{definition}\label{D:category}
     Fix a subset $\cZ \subseteq S^1$ satisfying axioms \textbf{Z1}-\textbf{Z3} and a field $\bK$. The {\em discrete cluster category with $N$ accumulation points over $\bK$} is the $\bK$-linear Krull-Schmidt category $\cC(\cZ)$ described by the following data:
     \begin{enumerate}
        \item{The indecomposable objects of $\cC(\cZ)$ are the arcs of $\cZ$.}
        \item{For indecomposable objects $X$ and $Y$ in $\cC(\cZ)$ we have
        \[
            \Hom_{\cC(\cZ)}(X,Y) \cong
                \begin{cases}
                    \bK & \text{if $X^-$ and $Y$ cross},
                    \\
                    0 & \text{otherwise}.
                \end{cases}
        \] 
        Note that $X^-$ and $Y$ cross if and only if we can write $X = \{x_0,x_1\}$ and $Y = \{y_0,y_1\}$ such that $x_0 \leq y_0 < x_1^-$ and $x_1 \leq y_1 < x_0^-$.
        }
        \item{Suppose there is a non-zero map $f\colon X \to Y$. Then $X^- =\{x_0^-,x_1^-\}$ and $Y = \{y_0,y_1\}$ cross, say with
        \[
            x_0^- < y_0 < x_1^- < y_1.
        \]
        Then the morphism $f$ factors through an indecomposable object $S = \{s_0,s_1\}$ if and only if we have inequalities
        \[
            x_0 \leq s_0 \leq y_0 \; \; \text{ and } \; \; x_1 \leq s_1 \leq y_1.
        \]}
     \end{enumerate}
    Whenever a subset $\{z_0,z_1\} \subseteq \cZ$ is not an arc of $\cZ$, i.e.\ if $z_1 \in \{z_0,z_0^+, z_0^-\}$, we set $\{z_0,z_1\} = 0$.
\end{definition}

\begin{remark}
    Up to triangulated equivalence, the category $\cC(\cZ)$ does not depend on the precise choice of $\cZ$, only on its number of proper accumulation points.
\end{remark}

\begin{remark}
    If we fix two arcs $X = \{x_0,x_1\}$ and $Y=\{y_0,y_1\}$ in $\cZ$, then we implicitly fix the ordering of their endpoints as points in $S^1$. Now, if there exists a non-zero morphism $f \colon X \rightarrow Y$ in $\cC(\cZ)$, then the arcs $X^-$ and $Y$ must cross, and so, precisely one of the inequalities
    \begin{equation} \label{S5:E:order1}
        x_0^- < y_0 < x_1^- < y_1
    \end{equation}
    or
    \begin{equation} \label{S5:E:order2}
        x_0^- < y_1 < x_1^- < y_0
    \end{equation}
    holds. If (\ref{S5:E:order1}) holds, then we say that $f$ is {\em enabled by the ordering (\ref{S5:E:order1}).} Otherwise, (\ref{S5:E:order2}) holds, and we say that $f$ is {\em enabled by the ordering (\ref{S5:E:order2}).}
\end{remark}

\begin{remark} \label{R:canonical_morphisms_exist}
    For all indecomposable objects $X$ and $Y$ in $\cC(\cZ)$, we may choose a canonical morphism $e_{XY} \colon X \to Y$ such that every morphism in $\Hom_{\cC(\cZ)}(X,Y)$ is of the form $\lambda e_{XY}$ for some $\lambda \in \bK$ and such that, if $e_{XZ}$ factors through $Y$, then $e_{YZ} \circ e_{XY} = e_{XZ}$.
\end{remark}

Igusa and Todorov \cite{ITcyclic} construct $\cC(\cZ)$ as a stable category of a Frobenius category, and therefore $\cC(\cZ)$ is triangulated. The suspension functor $\Sigma$ on $\cC(\cZ)$ is defined on indecomposable objects $X$ by $\Sigma X = X^-$. Moreover, for all objects $X,Y \in \cC(\cZ)$, there exist bi-functorial isomorphisms
\[
\Hom_{\cC(\cZ)}(X,Y) \cong D\Hom_{\cC(\cZ)}(Y,\Sigma^2X),
\]
where $D = \Hom_\bK(-,\bK)$,
and so $\cC(\cZ)$ is $2$-Calabi-Yau. In particular, we have
\begin{equation} \label{E:Ext1}
    \Ext^1_{\cC(\cZ)}(X,Y) \cong
        \begin{cases}
            \bK & \text{if $X$ and $Y$ cross}\\
            0 & \text{else}.
        \end{cases}
\end{equation}

Let $X = \{x_0,x_1\}$ and $Y=\{y_0,y_1\}$ be indecomposable objects in $\cC(\cZ)$ with a non-trivial morphism $f \colon X \rightarrow Y$ enabled by the ordering
\begin{equation*}
    x_0^- < y_0 < x_1^- < y_1.
\end{equation*}
Then there exists a distinguished triangle
\begin{center}
    \begin{tikzcd}
        X \arrow[r, "f"] & Y \arrow[r] & \{x_0^-,y_0\} \oplus \{x_1^-,y_1\} \arrow[r] & \Sigma X.
    \end{tikzcd}
\end{center}

The endpoints of an indecomposable object can be identified by considering morphisms to and from the ``shortest'' arcs in $\cC(\cZ)$. The arc $\fromarc{z} = \{z,z^{(2)}\}$ is the ``shortest arc starting in $z$'', in the sense that there are no arcs with $z$ as an endpoint, which have their second endpoint in the interval $[z,z^{(2)})$. Similarly, the arc $\toarc{z} = \{z^{(-2)},z\}$ is the ``shortest arc ending in $z$''.

 \begin{lemma} \label{L:arc_that_picks_out_points}
    Let $z$ be a point in $\cZ$ and let $W$ be an indecomposable object in $\cC(\cZ)$. \begin{enumerate}
        \item There exists a non-zero morphism $\fromarc{z} \rightarrow W$ if and only if $z$ is an endpoint of $W$.
        \item There exists a non-zero morphism $W \rightarrow \toarc{z}$ if and only if $z$ is an endpoint of $W$.
    \end{enumerate}
 \end{lemma}

 \begin{proof}
    There exists a morphism $\fromarc{z} \rightarrow W$ if and only if the arcs $\Sigma\fromarc{z}$ and $W$ cross, which is the case if and only if $W = \{w_0,w_1\}$ such that
 	\begin{equation*}
 		z^- < w_0 < z^+ < w_1.
 	\end{equation*}
    Or equivalently, we have $w_0 = z$. The second statement follows symmetrically.
 \end{proof}

Additive subcategories of $\cC(\cZ)$ that are closed under indecomposable summands correspond to sets of arcs of $\cZ$. In particular, to each subset of $\cZ$ we may assign a unique additive extension closed subcategory of $\cC(\cZ)$ via the {\em convex hull}.

\begin{definition}
	Let $S$ be a subset of $\cZ$. The {\em convex hull of $S$ in $\cC(\cZ)$} is the subcategory
	\begin{equation*}
		\cx(S) = \add \{ \{z,z'\} \text{ arc in } \cC(\cZ) \mid z,z' \in S \}.
	\end{equation*}
\end{definition}

\begin{remark} \label{R:convex_hull_ext_closed}
    The convex hull of a subset $S \subseteq \cZ$ is an extension closed subcategory of $\cC(\cZ)$ by \cite[Lemma~3.4]{Gratz--Zvonareva--2023--tstructures_in_clus_cat}.
\end{remark}

To each subcategory of $\cC(\cZ)$ we may also uniquely assign a subset of $\cZ$ via the {\em support}.

\begin{definition}
    Let $\cS$ be a subcategory of $\cC(\cZ)$. The {\em support} of $\cS$ denoted by $S(\cS)$ is the set of endpoints of arcs in $\cS$, i.e.
    \begin{equation*}
        S(\cS) = \{ z \in \cZ \mid \text{there exists a point } z' \in \cZ \text{ and an arc } \{z, z'\} \in \cS \}.
    \end{equation*}
\end{definition}

\subsection{T-structures on $\cC(\cZ)$}
 The t-structures on $\cC(\cZ)$ in the one-accumulation point case (i.e.\ when $N=1$) have been classified in \cite{Ng}. In the general case, they have been classified in \cite[Theorem~4.6]{Gratz--Zvonareva--2023--tstructures_in_clus_cat} building on the description of torsion pairs in \cite[Theorem~0.9]{GHJ}. The classification of t-structures is given by $\overline{\cZ}$-decorated non-crossing partitions $(\cP,\rx)$ whose definition we recall here. As before, we assume that $\cZ$ has $N \in \bN$ accumulation points and we set $[N] = \{1, \ldots, N\}$.

\begin{definition}
	A {\em non-crossing partition of $[N]$} is a partition $\cP = \{B_m \subseteq [N] \mid m \in I\}$ of $[N]$, such that whenever there exist indices $i, j, k, \ell \in [N]$ and $m_1,m_2 \in I$ with
		\begin{equation*}
			1 \leq i < k < j < \ell \leq N,
		\end{equation*}
		$i, j \in B_{m_1}$ and $k, \ell \in B_{m_2}$, then we must have $m_1 = m_2$.
    If $\{i\} \in \cP$ is a block in $\cP$, then we call $i$ a {\em singleton of $\cP$}. If $\{i,i+1\} \subseteq B$ for some block $B \in \cP$, then we call $i$ an {\em adjacency of $\cP$}.
\end{definition}

\begin{definition}\label{D:decoratednc}
	Let $\cP$ be a non-crossing partition of $[N]$. A {\em $\overline{\cZ}$-decoration of $\cP$} is an $N$-tuple $\rx = (x_1,x_2, \dots, x_N)$ such that
	\begin{equation*}
		x_i \in \begin{cases}
			[a_i,a_{i+1}) \text{ if $i$ is a singleton of $\cP$,}
			\\
			(a_i,a_{i+1}] \text{ if $i$ is an adjacency of $\cP$,}
			\\
			(a_i,a_{i+1}) \text{ else.}
		\end{cases}
	\end{equation*}
	A pair $(\cP,\rx)$ where $\cP$ is a non-crossing partition of $[N]$ and $\rx$ is a $\overline{\cZ}$-decoration of $\cP$ is called a {\em $\overline{\cZ}$-decorated non-crossing partition of $[N]$}.
\end{definition}

\begin{theorem}[{\cite[Theorem~4.6]{Gratz--Zvonareva--2023--tstructures_in_clus_cat}}] \label{T:GZ_classification_aisles}
Let $(\aisle,\coaisle)$ be a t-structure on $\cC(\cZ)$. There exists a unique $\overline{\cZ}$-decorated non-crossing partition $(\cP, \rx)$ such that the aisle $\aisle$ is given by
\begin{equation*}
	\aisle = \add \left( \bigcup_{B \in \cP} \cx \left( \bigcup_{i \in B} (a_i,x_i] \right) \right).
\end{equation*}
\end{theorem}

There is a similar description for the corresponding coaisle. By abuse of notation, for all indices $i \in [N]$, we set $a_i^- = a_i = a_i^+$.

\begin{corollary}[{\cite[Corollary~4.14]{Gratz--Zvonareva--2023--tstructures_in_clus_cat}}] \label{C:GZ_classification_coaisle}
Let $(\aisle,\coaisle)$ be a t-structure on $\cC(\cZ)$ with aisle corresponding to the $\overline{\cZ}$-decorated non-crossing partition $(\cP, \rx)$. Denote by $\cP^c$ the Kreweras complement of $\cP$. Then
\begin{equation*}
	\coaisle = \add \left( \bigcup_{B \in \cP^c} \cx \left( \bigcup_{i \in B} [x_i^-,a_{i+1}) \right) \right).
\end{equation*}
\end{corollary}

The coaisle of a t-structure is uniquely determined by the aisle and vice versa. Given the combinatorial description of the aisle $\aisle$, the coaisle $\coaisle$ of the corresponding t-structure is easy to read from the disc model, and vice versa. There is an example of this phenomena in Subsection~\ref{SS:EX:tstructures}, particularly in Figure~\ref{F:tstructure}. Explicitly, we have
\begin{eqnarray*}
    \aisle &=& \add\{Z \text{ arc of } \cZ \mid Z \text{ does not cross any arc in } \Sigma^{-1} \coaisle\},
    \\
    \coaisle &=& \add\{Z \text{ arc of } \cZ \mid Z \text{ does not cross any arc in } \Sigma \aisle\}.
\end{eqnarray*}

The equivalence classes of t-structures on $\cC(\cZ)$ can be easily determined from their combinatorial description.

\begin{definition}[{\cite[Definition~4.25]{Gratz--Zvonareva--2023--tstructures_in_clus_cat}}]
	For a $\overline{\cZ}$-decoration $\ry = (y_1, \ldots, y_N)$ we denote by $Z(\ry)$ its {\em set of $\cZ$-indices}
	\[
	Z(\ry) = \{i \in [N] \mid y_i \in \cZ\}.
	\]
	Two $\overline{\cZ}$-decorated non-crossing partitions $(\cP,\rx)$ and $(\cP',\rx')$ are {\em equivalent} if $\cP = \cP'$ and $Z(\rx) = Z(\rx')$.
\end{definition}

\begin{proposition}[{\cite[Proposition~4.27]{Gratz--Zvonareva--2023--tstructures_in_clus_cat}}]
	Two t-structures on $\cC(\cZ)$ are equivalent if and only if their respective associated $\overline{\cZ}$-decorated non-crossing partitions are equivalent.
\end{proposition}

In the computation of completions of $\cC(\cZ)$ with respect to coaisle metrics, it will be sufficient to restrict our attention to {\em right non-degenerate t-structures}. 

\begin{definition} \label{D:RNDG}
    A t-structure $(\aisle,\coaisle)$ on $\cT$ is {\em right non-degenerate} if $\bigcap_{n \in \bZ} \Sigma^n \coaisle = 0$. In which case, we say that the coaisle $\coaisle$ is also right non-degenerate. 
\end{definition}

Right non-degenerate t-structures on $\cC(\cZ)$ are precisely those that have no decorations $x_i$ such that $x_i=a_i$ \cite[Corollary~4.19]{Gratz--Zvonareva--2023--tstructures_in_clus_cat}. See, for example, the coaisles depicted in Subsection~\ref{SS:EX:RNDGcoaisles}, particularly in Figure~\ref{F:RNDGcoaisles}. Therefore, for a right non-degenerate t-structure $(\aisle, \coaisle)$ with $\overline{\cZ}$-decorated non-crossing partition $(\cP,\rx)$, the support of $\coaisle$ is given by 
\begin{equation*}
		S(\coaisle) = \{ x_i^{(s)} \mid x_i \in \cZ \text{ and } s \geq -1\}.
	\end{equation*}

Given an extension closed subcategory $\cS$ of $\cC(\cZ)$ there exists, up to equivalence, a unique largest aisle $\widehat{\aisle}(S)$ with respect to inclusion that lies in $\cS$. More precisely, there exists a unique equivalence class of aisles $[\widehat{\aisle}(S)]$ such that:
\begin{enumerate}
    \item there is a representative $\widehat{r}$ of $[\widehat{\aisle}(S)]$ with $\widehat{r} \subseteq \cS$ and;
    \item for any aisle $\aisle \subseteq S$, there exists a representative $r$ of $[\widehat{\aisle}(S)]$ with $\aisle \subseteq r$.
\end{enumerate}
The equivalence class is constructed as follows: Let $\aisle$ be an aisle of $\cC(\cZ)$ and denote by $[\aisle]$ its equivalence class under the equivalence of t-structures. Consider the finite set of equivalence classes
\[
    [A(\cS)] = \{[\aisle] \mid \aisle \subseteq \cS \text{ and $\aisle$ is an aisle in $\cC(\cZ)$} \}
\]
of aisles contained in $\cS$. For each $[\aisle] \in [A(\cS)]$ choose a representative $r_{[\aisle]} \in [\aisle]$ such that $r_{[\aisle]} \subseteq \cS$. By \cite[Proposition~5.12]{Gratz--Zvonareva--2023--tstructures_in_clus_cat}, t-structures on $\cC(\cZ)$ form a lattice under inclusion of aisles. Therefore, we may take the join over the representatives $r_{[\aisle]}$ to obtain an aisle
\[
    \widehat{\aisle}(\cS) = \bigvee_{[\aisle] \in [A(\cS)]} r_{[\aisle]}.
\]
In particular, suppose that $r_{[\aisle]}$ corresponds to the $\overline{\cZ}$-decorated non-crossing partition $(\cP(r_{[\aisle]}), \rx(r_{[\aisle]}))$ and let $\widehat{\cP}$ be the partition given by the join in the lattice of non-crossing partitions over the partitions $\cP_{r_{[\aisle]}}$. Then the aisle $\widehat{\aisle}(\cS)$ corresponds to the $\overline{\cZ}$-decorated non-crossing partition 
\begin{equation*}
    \left(\widehat{\cP}, \max\{\rx(r_{[\aisle]}) \mid [\aisle] \in [A(\cS)]\}\right),
\end{equation*}
where the maximum on the decorations is taken entry-wise. As $\cS$ is extension closed and contains every $r_{[\aisle]}$, it also contains $\widehat{\aisle}(\cS)$. By construction, $\widehat{\aisle}(\cS)$ is unique up to equivalence and we say that $[\widehat{\aisle}(\cS)]$ is the {\em equivalence class of largest aisles contained in $\cS$}. Note that not every representative of this equivalence class is contained in $\cS$; the precise decoration of an aisle plays a role. A representative $\widehat{\aisle}$ of $[\widehat{\aisle}(S)]$ is {\em a largest aisle contained in $\cS$} if $\widehat{\aisle} \subseteq \cS$. Such an example is depicted in Figure~\ref{F:largest_aisles}

\begin{lemma} \label{L:right-non-degenerate}
    Let $\coaisle$ be a coaisle of $\cC(\cZ)$ and let $\widehat{\aisle}$ be a largest aisle contained in $\coaisle$ with corresponding t-structure $(\widehat{\aisle},\widehat{\coaisle})$. Then for all integers $s$ we have $\Sigma^s  \widehat{\aisle} \subseteq \coaisle$, and, for all integers $r$, the intersection $\coaisle \cap \Sigma^r \widehat{\coaisle}$ is a right non-degenerate coaisle of $\cC(\cZ)$. 
\end{lemma}

\begin{proof}
    Let $(\aisle,\coaisle)$ be a t-structure on $\cC(\cZ)$ and let $(\cP,\rx)$ and $(\widehat{\cP}, \widehat{\rx})$ be the $\overline{\cZ}$-decorated non-crossing partitions associated to $\aisle$ and $\widehat{\aisle}$ respectively. As $\widehat{\aisle} \subseteq \coaisle$, for each block $\widehat{B} \in \widehat{\cP}$ there exists a block $B \in \cP$ with an inclusion
    \begin{equation*}
	\cx \left( \bigcup_{i \in \widehat{B}} (a_i,\widehat{x}_i] \right) \subseteq \cx \left( \bigcup_{j \in B} [x_j^-,a_{j+1}) \right),
    \end{equation*}
    so we have $\widehat{B} \subseteq B$. Also for each index $i \in [N]$ we have
    \begin{equation*}
        (a_i,\widehat{x}_i] \subseteq [x_i^-,a_{i+1}),
    \end{equation*}
    so $\widehat{x}_i \neq a_i$ implies $x_i = a_i$. Thus, for all integers $s$, we have
    \begin{equation*}
        (a_i,\widehat{x}_i^{(-s)}] \subseteq [x_i^-,a_{i+1}),
    \end{equation*}
    and $\Sigma^s \widehat{\aisle} \subseteq \coaisle$. Now fix an integer $r$ and consider the intersection $\rcoaisle = \coaisle \cap \Sigma^r\widehat{\coaisle}$, which is a coaisle of $\cC(\cZ)$ by \cite[Remark~5.4]{Gratz--Zvonareva--2023--tstructures_in_clus_cat}. As $\widehat{\aisle}$ is a largest aisle contained in $\coaisle$, for each index $i \in [N]$, we have $x_i = a_i$ implies $\widehat{x}_i \neq a_i$. Thus, it follows that
    \begin{equation*}
        S(\rcoaisle) \cap (a_i,a_{i+1}) = [x_i^-,a_{i+1}) \cap [\widehat{x}_i^{-(r+1)},a_{i+1}) \neq (a_i,a_{i+1}).
    \end{equation*}
    Hence, $\rcoaisle$ is a right non-degenerate coaisle.
\end{proof}

An example of this lemma can be found in Figure~\ref{F:largest_aisles}.

\subsection{Double fans}\label{S:double fans}

In this subsection we focus on a special collection of sequences in $\cC(\cZ)$, called double fans, which control the coaisle metric completions of $\cC(\cZ)$.

\begin{definition} \label{D:double_fan}
    Let $Z = \{z_0,z_1\}$ be an arc of $\cZ$. For all integers $n>0$, define points $f_n$ and $f'_n$ of $\cZ$ as follows:
    \begin{itemize}
        \item Either for all $n >0$ we have $f_n = z_0$, or for all $n > 0$ we have $f_n=z_0^{(n)}$;
        \item Either for all $n >0$ we have $f'_n = z_1$, or for all $n > 0$ we have $f'_n=z_1^{(n)}$.
    \end{itemize}
    
    Assume that for all $n > 0$ the set $F_n = \{f_n,f'_n\}$ is an arc in $\cC(\cZ)$, and that the limits $f = \lim f_n$ and $f' = \lim f'_n$ in $\overline{\cZ}$ are distinct.
    Let $\zeta_n \colon F_n \rightarrow F_{n+1}$ be the canonical morphism enabled by the ordering
	\begin{equation*}
		f_n^- < f_{n+1} < {f'_n}^- < f'_{n+1}.
	\end{equation*}
	Then the sequence $\mathbf{F} =  (F_n,\zeta_n)_{n>0}$ is an {\em indecomposable double fan} and we denote its module colimit by
\[
    \mocolim \mathbf{F} = \{f, f'\}.
\]
A direct sum 
\[
    \mathbf{F} = (\bigoplus_{d=1}^m F_n(d), \bigoplus_{d=1}^m\zeta_n(d))_{n > 0}
\]
of indecomposable double fans $(F_n(d),\zeta_n(d))_{n>0}$ is called a {\em double fan}.
\end{definition}

We use the name double fan due to the appearance of these sequences in the combinatorial model of $\cC(\cZ)$, as can be seen in Figure~\ref{F:doublefanseq}. Since the topological closure of $\cZ$ is $\cZ \cup L(\cZ)$ we may realise the module colimits of double fans in the combinatorial model of $\cC(\cZ)$ by allowing arcs to end at accumulation points. In line with that viewpoint we call a two-element subset 
$\{z_0,z_1\}$ of $\overline{\cZ}$ an {\em arc of $\overline{\cZ}$} if $z_0 \notin \{z_1^-,z_1,z_1^+\}$; where for any accumulation point $a_i \in L(\cZ)$, by abuse of notation, we write $a_i^- = a_i = a_i^+$.

\begin{remark}\label{R:every arc is a mocolim}
Note that every arc $\{f,f'\}$ of $\overline{\cZ}$ can be realised as a module colimit of a double fan $\mathbf{F} = (\{f_n,f'_n\},\zeta_n)_{n>0}$: If $f \in \cZ$, we set $f_n = f$ for all $n>0$. Conversely, if $f = a_k \in L(\cZ)$ for some index $k \in [N]$, then we choose a point $z \in (a_{k-1}, a_k) \cap \cZ$ and set $f_n = z^{(n)}$ for all $n>0$. Analogously, we define $f'_n$.
\end{remark}

The morphisms between module colimits of double fans can be described via the combinatorial model.

\begin{lemma} \label{L:morphisms_between_double_fans}
Let $\mathbf{F}$, $\mathbf{G}$ and $\mathbf{H}$ be indecomposable double fans with module colimits $F$, $G$ and $H$ respectively. Then we have
    \begin{equation*}
		\Hom_{\Mod\cC(\cZ)}(F, G) \cong \begin{cases}
			\bK, &\text{if } F = \{f, f'\} \text{ and }  G = \{g,g'\} \text{ with } \\
    & f \leq g < {f'}^- \text{ and } f' \leq  g' < f^-
			\\
			0, &\text{ otherwise.}\\
		\end{cases}
	\end{equation*}
Moreover, let $\varphi \colon F = \{f,f'\} \to G = \{g,g'\}$ and $\psi \colon G \to H = \{h,h'\}$ be non-trivial maps with
\[
    f \leq g < {f'}^-, \;\; f' \leq g' < f^- \text{ and } g \leq h < {g'}^-, \;\; g' \leq h' < g^-. 
\]
Then $\psi \circ \varphi \neq 0$ if and only if 
\[
    f \leq g \leq h \text{ and } f' \leq g' \leq h'.
\]
\end{lemma}

\begin{proof}
We apply the Yoneda Lemma to obtain an isomorphism
	\begin{align*}
		\Hom_{\Mod\cC(\cZ)}(F,G) &\cong \lim G(\mathbf{F}) \\ &= \lim \colim \Hom_{\cC(\cZ)}(\mathbf{F}, \mathbf{G}).
	\end{align*}
Therefore, $\Hom_{\Mod\cC(\cZ)}(F,G)$ is either trivial or isomorphic to $\bK$. Fix $\mathbf{F} = (F_i,\widehat{f}_i)_{i > 0}$ and $\mathbf{G} = (G_j,\widehat{g}_j)_{j > 0}$ with colimit maps $\lambda_i \colon \yo(F_i) \to F$ and $\mu_j \colon \yo(G_j) \to G$. 

Consider a morphism $\varphi \colon F \to G$ in $\Mod \cC(\cZ)$, and consequently, consider, for all $i>0$, the morphism $\varphi \circ \lambda_i \colon \yo(F_i) \to G$. The module colimit $G$ is a filtered colimit and, for all $i>0$, $\yo(F_i)$ is compact in $\Mod\cC(\cZ)$. Therefore, for all $i>0$, there exists an $M_i > 0$ such that for all $j \geq M_i$ there exists a morphism $\varphi_{i,j} \colon F_i \to G_j$ and the following hold:
\begin{enumerate}
    \item For all $i>0$ and $j \geq M_i$, there exists a commutative diagram
    \begin{center}
    \begin{equation} \label{E:piecewise}
        \begin{tikzcd}[column sep={8em,between origins}]
            F \arrow[r, "\varphi"] & G
            \\
            \yo(F_{i}) \arrow[u, "\lambda_i"] \arrow[r, "\varphi_{i,j} \circ -",swap] & \yo(G_{j}). \arrow[u, "\mu_{j}",swap]
        \end{tikzcd}
    \end{equation}
    \end{center}
    \item For all $i>0$ and $j'\geq j \geq M_i$, we have $\varphi_{i,j'} = g_{j,j'} \circ \varphi_{i,j}$.
    \item The module colimit $G$ is a cone of $\yo(\mathbf{F})$ with cone maps $\varphi \circ \lambda_i = \mu_{M_i} \circ (\varphi_{i,M_i} \circ -) \colon \yo(F_i) \rightarrow G$.
\end{enumerate}
In particular, if $\varphi$ is non-trivial, then, for all $i \gg 0$, $\varphi \circ \lambda_i$ is non-zero, and so, for all $j \geq M_i$, $\varphi_{i,j}$ is also non-trivial.

Conversely, assume for all $i \gg 0$ there exists an $M_i > 0$ such that for all $j \geq M_i$ there exists a non-zero morphism $\varphi_{i,j} \colon F_i \to G_j$ that satisfies conditions (ii) and (iii). 
Then by (iii) and the universal property of $F$ there exists a morphism $\varphi \colon F \to G$ that satisfies condition (i). Fix an $i \gg 0$. For all $j' \geq j \geq M_i$, the morphism $\varphi_{i,j'}$ is non-trivial, and so, $\mu_j \circ \varphi_{i,j} \circ \mathrm{id}_{F_i}$ is also non-trivial. Therefore, by evaluating the diagram in (\ref{E:piecewise}) at $F_i \in \cC(\cZ)$ and following the identity $\mathrm{id}_{F_i}$ through the evaluated diagram, we see that $\varphi_{F_i}$ is non-trivial, and so, $\varphi$ is also non-trivial.

 To summarise, there exists a non-trivial map $\varphi \colon F \to G$ if and only if for all $i \gg 0$ there exists an $M_i > 0$ such that for all $j \geq M_i$ there exists a non-trivial map $\varphi_{i,j} \colon F_i \to G_j$ that satisfies conditions (ii) and (iii). 
 This is equivalent to the following: We can write $F_i = \{f_i,f'_i\}$ and $G_j = \{g_j,g'_j\}$ such that for all $i' \geq i \gg 0$ there exist $M_{i'} \geq M_i > 0$ such that for all $j' \geq j \geq M_{i'}$ we have
\begin{equation}\label{E:small_crossings}
    f_i \leq f_{i'} \leq g_j \leq g_{j'} < {f'_{i}}^- \leq {f'_{i'}}^- \text{ and } f'_i \leq f'_{i'} \leq g'_{j} \leq g'_{j'} < f_i^- \leq f_{i'}^-.
\end{equation}
 Let $f = \lim f_i$, $f' = \lim f'_i$ and $g = \lim g_j$, $g' = \lim g'_j$ in $\overline{\cZ}$. Then $F = \{f,f'\}$ and $G = \{g,g'\}$ and the inequality (\ref{E:small_crossings}) is satisfied if and only if
 \begin{equation}\label{E:big_crossing}
     f \leq g < {f'}^- \text{ and } f' \leq g' < f^-.
\end{equation}
This proves the first part of the claim.

Let now $\varphi \colon F \to G$ and $\psi \colon G \to H$ be non-trivial maps. As before this is equivalent to the following:
For all $i,j \gg 0$ there exist $M_i, N_j > 0$ such that for all $\ell \geq M_i$ and $k \geq N_j$ there exist non-trivial morphisms $\varphi_{i,\ell} \colon F_i \rightarrow G_\ell$ and $\psi_{j,k} \colon G_j \to H_k$ that satisfy conditions (ii) and (iii) (with appropriate renaming of indices and morphisms). Without loss of generality we may assume that $j \geq M_i \gg 0$, and so, for all $i \gg 0$, $j \geq M_i \gg 0$ and $k \geq N_j$, we obtain a commutative diagram
\begin{center}
\begin{equation}\label{E:piecewise_composition}
        \begin{tikzcd}
            F \arrow[r, "\varphi"] & G \arrow[r, "\psi"] & H
            \\
            \yo(F_i) \arrow[r, "\varphi_{i,j} \circ -",swap] \arrow[u, "\lambda_i"] & \yo(G_j) \ar[u,"\mu_j"]\arrow[r, "\psi_{j,k} \circ -",swap] & \yo(H_k) \arrow[u, "\nu_k"].
        \end{tikzcd}
    \end{equation}
    \end{center}
    Therefore, by the previous argument, a morphism $\psi \circ \varphi$ is non-trivial if and only if for all $i \gg 0$, $j \geq M_i \gg 0$ and $k \geq N_j$, the morphism $\psi_{j,k} \circ \varphi_{i,j}$ is non-trivial.

Assume now that $F = \{f,f'\}$, $G = \{g,g'\}$ and $H = \{h,h'\}$, and that the non-trivial maps $\varphi$ and $\psi$ are enabled by the inequalities
\begin{equation}\label{E:phi-crossing}
    f \leq g < {f'}^- \text{ and } f' \leq g' < f^-,
\end{equation}
respectively
\begin{equation}\label{E:psi-crossing}
    g \leq h < {g'}^- \text{ and } g' \leq h' < g^-.
\end{equation}
For all integers $m >0$ we can write $F_m = \{f_m,f'_m\}$, $G_m = \{g_m,g'_m\}$ and $H_m = \{h_m,h'_m\}$ with $f = \lim f_m$, $f' = \lim f'_m$, etc. Then for all $i \gg 0$ and $j \geq M_i \gg 0$ the non-trivial map $\varphi_{i,j} \colon F_i \to G_j$ must be enabled by the inequality
\[
    f_i^- < g_j < {f'_i}^- < g'_j
\]
and for all $j \geq M_i \gg 0$ and $k \geq N_j$ the non-trivial map $\psi_{j,k} \colon G_j \to H_k$ must be enabled by the inequality
\[
    g_j^- < h_k < {g'_j}^- < h'_k.
\]
Now $\psi \circ \varphi \neq 0$ is equivalent to $\psi_{j,k} \circ \varphi_{i,j} \neq 0$ for all $i \gg 0$, $j \geq M_i \gg 0$ and $k \geq N_j$. Choosing $M_i$ and $N_j$ large enough, this is equivalent to
\[
    f_i \leq g_j \leq h_k \text{ and } f'_i \leq g'_j \leq h'_k
\]
for all $i \gg 0$, $ j\geq M_i \gg 0$ and $k \geq N_j$. Since all our sequences $\{f_n\}_{n >0}$, $\{f'_n\}_{n>0}$, etc.\ are either constant or of the form $\{z^{(n)}\}_{n > 0}$ for some point $z \in \cZ$ we obtain that, as desired, this is equivalent to  
\[
    f \leq g \leq h \text{ and } f' \leq g' \leq h'.
\]
\end{proof}

In Subsection~\ref{SS:EX:morphismsbetweendoublefans}, there is an example concerning morphisms between module colimits of double fans.

\section{Coaisle metric completions of discrete cluster categories} \label{S:coaisle metric}
 In this section we explicitly compute the completions of $\cC(\cZ)$ with respect to all coaisle metrics and show that, unlike for aisle metrics, in general these completions are not subcategories of $\cC(\cZ)$. We start by defining a combinatorial completion of $\cC(\cZ)$ with respect to a fixed coaisle that we will prove is equivalent to the coaisle metric completion. 

Throughout Section~\ref{S:coaisle metric}, we fix a t-structure $(\aisle,\coaisle)$ with corresponding $\overline{\cZ}$-decorated non-crossing partition $(\cP,\rx)$, where $\rx = (x_1, \ldots, x_N)$. Recall that the coaisle metric $\cM_\coaisle$ is given by
\[
   (\cM_\coaisle)_0 = \cC(\cZ) \text{ and } (\cM_\coaisle)_t = \Sigma^{-t} \coaisle \text{ for all $t >0$.}
\]

We will show that the (indecomposable) objects in the completion of $\cC(\cZ)$ with respect to any coaisle metric correspond to module colimits of (indecomposable) double fans. We know from Section~\ref{S:double fans} that these module colimits can be interpreted combinatorially as arcs of $\overline{\cZ}$. With this motivation, for a subset $\overline{S} \subseteq \overline{\cZ}$ we define the {\em completed convex hull} to be the full subcategory of $\Mod \cC(\cZ)$ given by 
\[
    \overline{\cx}(\overline{S}) = \add \{F \in \Mod \cC(\cZ) \mid F \cong \{z_0,z_1\}  \text{ is an arc 
 of }\overline{\cZ} \text{ with } z_0,z_1 \in \overline{\cS} \}.
\]

Let $B$ be a block in the non-crossing partition $\cP$. We define the {\em $\cZ$-support of $B$} as
\begin{equation*}
    \cZ_B = \bigcup_{\substack{ i \in B, \\ x_i \neq a_i}} (a_i,a_{i+1}) \cap \cZ,
\end{equation*}
and the {\em $\overline{\cZ}$-completion of $B$} as
\begin{equation*}
    \overline{\cZ}_B = \cZ_B \cup \{a_i \mid i \in B, a_i\notin \{x_1, \ldots, x_N\}\}.
\end{equation*}

\begin{definition}
The {\em combinatorial completion of $\cC(\cZ)$ with respect to the $\overline{\cZ}$-decorated non-crossing partition $(\cP,\rx)$} is the full replete subcategory $\mathfrak{R}_{(\cP,\rx)}$ of $\Mod \cC(\cZ)$ with objects
\[
     \add \left( \bigcup_{B \in \cP} \overline{\cx} \left( \overline{\cZ}_B \right) \right).
\]
\end{definition}

An example of the combinatorial completion of a coaisle can be found in Subsection~\ref{SS:EX:completion}.

\begin{remark}
    The $\overline{\cZ}$-completion of a block includes only the accumulation points of $\cZ$ that do not appear in the decoration $\rx$. In Lemma~\ref{L:double_fan_cauchy}, we prove that double fans converging to an accumulation point $a_i$ such that $a_i = x_{i+1}$ are not Cauchy sequences and in Lemma~\ref{L:double_fan_cptly_supp} we prove that double fans converging to an accumulation point $a_i$ such that $a_i = x_i$ are not compactly supported.
\end{remark}

\begin{remark} \label{R:completion_like_union_of_aisles}
    The combinatorial completion $\cR_{(\cP,\rx)}$ can be thought of as a combinatorial completion of $\bigcup_{p \in \bZ} \Sigma^p\aisle$. Indeed, we have
    \begin{equation*}
        \bigcup_{p \in \bZ} \Sigma^p\aisle = \add \left( \bigcup_{B \in \cP} \cx \left( \cZ_B \right) \right).
    \end{equation*}
    In particular, the indecomposable objects in $\mathfrak{R}_{(\cP,\rx)}$ that do not lie in $\bigcup_{p \in \bZ} \Sigma^p \aisle$ are precisely the arcs of $\overline{\cZ}$ of the form $\{a_i,z\}$, where $i \in B \in \cP$ and $z \in \bigcup_{j \in B} [a_j,a_{j+1})$.
\end{remark}

To distinguish it clearly from the a priori distinct combinatorial completion of $\cC(\cZ)$, we refer to $\fS_{\cM_\coaisle}$ as the {\em metric completion of $\cC(\cZ)$} (with respect to the coaisle metric $\cM_\coaisle$).

\begin{theorem}\label{T:completions agree}
    Let $(\aisle,\coaisle)$ be a t-structure on $\cC(\cZ)$ with corresponding $\overline{\cZ}$-decorated non-crossing partition $(\cP,\rx)$. The metric completion of $\cC(\cZ)$ with respect to $\cM_\coaisle$ is equivalent to its combinatorial completion with respect to $(\cP,\rx)$:
    \[
        \fS_{\cM_\coaisle} = \mathfrak{R}_{(\cP,\rx)}.
    \]
\end{theorem}

The rest of Section~\ref{S:coaisle metric} is dedicated to proving Theorem~\ref{T:completions agree}. 

\subsection{The combinatorial completion as a subcategory of the metric completion} \label{subsec:comb_in_met}

In this section we prove that every arc that lies in the combinatorial completion of a coaisle can be realised as the module colimit of a compactly supported Cauchy sequence with respect to the corresponding coaisle metric. To do so we utilise double fans. 

\begin{lemma} \label{L:double_fan_cauchy}
 Let $\mathbf{F} = (F_n,\zeta_n)_{n >0}$ be an indecomposable double fan with $\mocolim \mathbf{F} = \{f,f'\}$. Fix indices $i,j \in [N]$ such that $f \in (a_{i},a_{i+1}]$ and $f' \in (a_{j},a_{j+1}]$. Then $\mathbf{F}$ is a Cauchy sequence with respect to the coaisle metric $\cM_\coaisle$ if and only if the following hold:
	\begin{enumerate}
		\item $x_{i} \neq a_{i+1}$ or $f \in \cZ$;
		\item $x_{j} \neq a_{j+1}$ or $f' \in \cZ$.
	\end{enumerate}
\end{lemma}

\begin{proof}
For all integers $m >0$ set $F_m = \{f_m,f'_m\}$ such that the sequences $(f_n)_{n > 0}$ and $(f'_n)_{n>0}$ are either constant or of the form $(z^{(n)})_{n>0}$ for some point $z \in \cZ$, and let $f = \lim f_m$ and $f' = \lim f'_m$ in $\overline{\cZ}$.
   For all integers $m'>m>0$ the cone of $\zeta_{m,m'} \colon F_m \rightarrow F_{m'}$ is isomorphic to
 \[
    \{f_m^-,f_{m'}\} \oplus \{{f'_m}^-, {f'_{m'}}\}.
\]

    Suppose that $\mathbf{F}$ is a Cauchy sequence. Then without loss of generality by taking a truncation if necessary, we may assume for all integers $m'>m>0$ the cone of $\zeta_{m,m'}$ lies in $\coaisle$. In particular, the object $\{f_m^-,f_{m'}\}$ is an arc in $\coaisle$ or it is trivial. If the object is an arc in $\coaisle$, then both $f_m^-$ and $f_{m'}$ lie in the support $S(\coaisle)$, and the intersection
    \begin{equation*}
        (a_i,a_{i+1}) \cap S(\coaisle) = [x_i^-,a_{i+1})
    \end{equation*}
    is non-empty, i.e.\ $x_i \neq a_{i+1}$. Otherwise the object is trivial and we have $f_m = f_{m'}$, i.e.\ the sequence $(f_n)_{n>0}$ is constant which implies $f \in \cZ$. This proves (i) and (ii) follows similarly.

    To prove the converse, for each integer $t\geq 0$ we wish to find an integer $M>0$ such that for all $m'>m\geq M$ both summands of the cone of $\zeta_{m,m'}$ lie in $\Sigma^{-t} \coaisle$. Condition (i) ensures that this is true for the object $\{f_m^-,f_{m'}\}$, and condition (ii) ensures the same for the object $\{{f'_m}^-, {f'_{m'}}\}$.
\end{proof}

\begin{lemma} \label{L:double_fan_cptly_supp}
Let $\mathbf{F} = (F_n,\zeta_n)_{n >0}$ be an indecomposable double fan and let $F = \{f,f'\}$ be its module colimit. Fix indices $k,\ell \in [N]$ such that $f \in [a_k,a_{k+1})$ and $f' \in [a_\ell,a_{\ell+1})$. Then $\mathbf{F}$ is compactly supported with respect to the coaisle metric $\cM_\coaisle$ if and only if the following hold:
 \begin{enumerate}
    \item $k$ and $\ell$ lie in the same block of $\cP$;
    \item $x_k \neq a_k$;
    \item $x_\ell \neq a_\ell$.
 \end{enumerate}
\end{lemma}

\begin{proof}
    The double fan $\mathbf{F}$ is compactly supported if and only if there exists an integer $t \geq 0$ such that for all arcs $Y$ in $\Sigma^{-t}\coaisle$ we have 
    \[
        0 = \mocolim\mathbf{F}(Y) \cong \Hom_{\Mod\cC(\cZ)}(\yo(Y), F).
    \]
     By Lemma~\ref{L:morphisms_between_double_fans} this is the case if and only if there exists an integer $t \geq 0$ such that $F$ crosses no arcs in 
     \[
     \Sigma^{-t+1}\coaisle = \add \left( \bigcup_{B \in \cP^c} \cx \left( \bigcup_{i \in B} [x_i^{-t},a_{i+1}) \right) \right).
     \]
     If $k \neq \ell$, then by construction of the Kreweras complement, this is the case if and only if $k$ and $\ell$ lie in the same block of $\cP$. In this case, by definition of $\overline{\cZ}$-decorated partitions we have $x_k \neq a_k$ and $x_\ell \neq a_\ell$. On the other hand, if $k = \ell$, then automatically $k$ and $\ell$ lie in the same block of $\cP$ and there exists an integer $t \geq 0$ such that $F$ crosses no arcs in $\Sigma^{-t}\coaisle$ if and only if $x_k \neq a_k$.
\end{proof}

\begin{proposition} \label{P:comb_comp_equals_double_fan_mocolims}
    An indecomposable object in $\Mod\cC(\cZ)$ lies in the combinatorial completion $\fR_{(\cP,\rx)}$ if and only if it is isomorphic to the module colimit $\{f,f'\}$  of an indecomposable double fan that is Cauchy and compactly supported with respect to the coaisle metric $\cM_\coaisle$.
\end{proposition}

\begin{proof}
    First suppose that an indecomposable object $F$ in $\Mod \cC(\cZ)$ is isomorphic to the module colimit $\{f,f'\}$ of an indecomposable double fan $\mathbf{F}$ that is Cauchy and compactly supported with respect to $\cM_\coaisle$. Fix indices $k,\ell \in [N]$ such that $f$ and $f'$ lie in $[a_k,a_{k+1})$ and $[a_\ell,a_{\ell+1})$ respectively. As $\mathbf{F}$ is compactly supported, by Lemma~\ref{L:double_fan_cptly_supp}, the indices $k$ and $\ell$ lie in the same block $B \in \cP$ with $x_k \neq a_k$ and $x_\ell \neq a_\ell$. If $f$ is not a proper accumulation point then $f$ lies in $(a_k,a_{k+1}) \subseteq \overline{\cZ}_B$. If on the other hand $f$ is a proper accumulation point, then because $\mathbf{F}$ is Cauchy, by Lemma~\ref{L:double_fan_cauchy} we also have $x_{k-1} \neq a_k$, and so $f=a_k \in \overline{\cZ}_B$. Thus, in either case we have $f \in \overline{\cZ}_B$ and symmetrically $f' \in \overline{\cZ}_B$. Therefore $F \cong \{f,f'\}$ lies in $\mathfrak{R}_{(\cP,\rx)}$.

     Conversely, every indecomposable object in $\cR_{(\cP,\rx)}$ is isomorphic to an indecomposable object in  $\add \left( \bigcup_{B \in \cP} \overline{\cx} \left( \overline{\cZ}_B \right) \right)$. Therefore, it is isomorphic to an arc $\{f,f'\}$ of $\overline{\cZ}_B$ for some block $B \in \cP$. Hence there exist indices $k, \ell \in B$ such that $f$ and $f'$ lie in $[a_k,a_{k+1})$ and $[a_\ell,a_{\ell+1})$ respectively, $x_k \neq a_k$ and $x_\ell \neq a_\ell$. 
     By Remark~\ref{R:every arc is a mocolim} there exists an indecomposable double fan $\mathbf{F} = (F_n,\zeta_n)_{n>0}$ with module colimit $\{f,f'\}$, and by Lemma~\ref{L:double_fan_cptly_supp} it is compactly supported. 
     If $f \notin \cZ$, then $f = a_k \in \overline{\cZ}_B$ and by definition of the combinatorial completion, we must have $a_k \neq x_{k-1}$. Symmetrically, if $f' \notin \cZ$ we must have $a_\ell \neq x_{\ell-1}$. It follows by Lemma~\ref{L:double_fan_cauchy} that $\mathbf{F}$ is also Cauchy.    
\end{proof}

One direction of Theorem~\ref{T:completions agree} is an immediate consequence of Proposition~\ref{P:comb_comp_equals_double_fan_mocolims}.

\begin{corollary} \label{C:comb_lies_in_metric}
	The combinatorial completion $\mathfrak{R}_{(\cP,\rx)}$ of $\cC(\cZ)$ with respect to the $\overline{\cZ}$-decorated non-crossing partition $(\cP,\rx)$ is a subcategory of the metric completion $\fS_{\cM_\coaisle}$ of $\cC(\cZ)$ with respect to the coaisle metric $\cM_\coaisle$.
\end{corollary}

Now we focus on the opposite direction of Theorem~\ref{T:completions agree}. Note that, given Proposition~\ref{P:comb_comp_equals_double_fan_mocolims}, it suffices to prove that every indecomposable object in the metric completion of $\cC(\cZ)$ is isomorphic to the module colimit of a compactly supported Cauchy indecomposable double fan.
\subsection{The metric completion as a subcategory of the combinatorial completion} \label{subsec:met_in_comb}

To prove that the metric completion with respect to $\cM_\coaisle$ sits inside the combinatorial completion we compute the module colimits of compactly supported Cauchy sequences with respect to the coaisle metric $\cM_\coaisle$. Given a compactly supported stabilising sequence $\mathbf{E}$ with respect to a suitable good metric (in particular a coaisle metric) we can compute its module colimit via an approximation by a suitable t-structure.

\subsubsection{Utilising right non-degenerate t-structures}

Passing to a suitable right non-degenerate t-structure allows us to simplify a large portion of our computations.

\begin{proposition} \label{P:reduce_to_RNDG_coaisle_sequences}
    Let $\widehat{\aisle}$ be a largest aisle in $\coaisle$ with t-structure $(\widehat{\aisle},\widehat{\coaisle})$. Then $\rcoaisle = \coaisle \cap \Sigma^{2(t+1)}\widehat{\coaisle}$ is a coaisle of a right non-degenerate t-structure $(\raisle,\rcoaisle)$ of $\cC(\cZ)$. In particular, for all objects $E \in \fS_{\cM_{\coaisle}}(\cT)$ that are compactly supported at $t\geq0$ there exists a sequence $\mathbf{\widehat{Y}}$ in $\Sigma^t\widehat{\coaisle}$ such that 
        \[
            E \cong \mocolim \widehat{\mathbf{Y}},
        \]
    and the following hold:
        \begin{enumerate}
            \item the $(\Sigma^{-(t-1)}\raisle)$-approximation of $\mathbf{\widehat{Y}}$ stabilises at $0$;
            \item the $(\Sigma^{-(t-1)}\rcoaisle)$-approximation of $\mathbf{\widehat{Y}}$ is a minimal coaisle sequence with respect to $\Sigma^{-(t-1)}\rcoaisle$.
        \end{enumerate}
\end{proposition}

\begin{proof}
    By Lemma~\ref{L:right-non-degenerate} the intersection $\rcoaisle = \coaisle \cap \Sigma^{2(t+1)}\widehat{\coaisle}$ is the coaisle of a right non-degenerate t-structure of $\cC(\cZ)$. So, the statement follows by Proposition~\ref{P:combine_aisle_and_coaisle} with $\cM = \cM_{\coaisle}$.
\end{proof}

In light of Proposition~\ref{P:reduce_to_RNDG_coaisle_sequences}, to compute the module colimits of compactly supported Cauchy sequences with respect to the coaisle metric, we study coaisle sequences with respect to right non-degenerate coaisles.

\subsubsection{Coaisle sequences}

For this subsection, we assume that $(\aisle, \coaisle)$ is right non-degenerate. In particular, every point in the support $S(\coaisle)$ is of the form $x_i^{(s)}$ for some $Z$-index $i \in Z(\rx)$ and some integer $s \geq -1$ .

\begin{definition}
    Let $B = \{i_1,i_2,\dots,i_n\}$ be a block of a partition $\cQ$ of $[N]$. Suppose that we have
    \begin{equation*}
        i_1 < i_2 < \dots < i_n
    \end{equation*}
    with respect to the natural cyclic order on $[N]$. Fix an integer $m$. Then the {\em $\cQ$-predecessor} of $i_m$ is $i_{m-1}$ and the {\em $\cQ$-successor} of $i_m$ is $i_{m+1}$ where the indices are taken modulo $n$.
\end{definition}

\begin{remark}
    Note that an index $i$ is its own $\cQ$-predecessor if and only if it is its own $\cQ$-successor if and only if $i$ is a singleton in the partition $\cQ$.
\end{remark}

The goal of this subsection is to prove the following result, which shows that minimal coaisle sequences look similar to double fans. 

\begin{proposition} \label{P:final_classification_coaisle_sequence}
    Assume $(\aisle, \coaisle)$ is right non-degenerate. Let $\mathbf{Y} = (Y_n,g_n)_{n>0}$ be a minimal coaisle sequence with respect to $\coaisle$. Then $\mathbf{Y}$ has a subsequence $\mathbf{Y}_\mathbf{I} = (Y'_n,g'_n)_{n>0}$ for which the following hold:
     \begin{enumerate}
        \item For all integers $n>0$, each indecomposable direct summand of $Y'_n$ is of the form $\arc{i}{j}{s}$ for some $Z$-indices $i,j \in Z(\rx)$ where $j$ is the $\cP^c$-predecessor of $i$ and for some integer $s > n$;  
        \item For all integers $n'>n>0$, the
        cone of $g'_{n,n'}$ lies in
        \begin{equation*}
            \add \left( \bigcup_{j \in Z(\rx)} \cx \left((x_j^{(n-1)},a_{j+1})\right) \right). 
        \end{equation*} In particular $\mathbf{Y_I}$ is Cauchy with respect to $\cM_\coaisle$.
     \end{enumerate}
      
 \end{proposition}

To prove Proposition~\ref{P:final_classification_coaisle_sequence}, we break it down into a series of helpful results. 

 \begin{lemma} \label{L:factoring_crossings}
    Let $\cM = \{B_t\}_{t\geq0}$ be a good metric on $\cC(\cZ)$. Let $\mathbf{E} = (E_n,f_n)_{n>0}$ be a sequence in $\cC(\cZ)$ that is compactly supported at $t \geq 0$ with respect to $\cM$ and fix an object $Z \in B_t$. Then, for each integer $m>0$, there exists an integer $m_Z > m$ such that for each indecomposable summand $\{v_0,v_1\}$ of $E_m$ and $\{w_0,w_1\}$ of $E_{m_Z}$ the following holds: If the composition
    \begin{center}
      \begin{tikzcd}
          \{v_0,v_1\} \arrow[r, "\iota"] & E_{m} \arrow[r, "f_{m,m_Z}"] & E_{m_Z} \arrow[r, "\pi"] & \{w_0,w_1\}
      \end{tikzcd}
  \end{center}
  (where $\iota$ and $\pi$ denote the canonical inclusion and projection of indecomposable summands respectively), is non-zero and enabled by the ordering
 	\begin{equation*}
 		v_0^- < w_0 < v_1^- < w_1,
 	\end{equation*}
 	then a non-trivial indecomposable summand $U$ of $Z$ cannot be of the form $U = \{u_0,u_1\}$ where
 	\[
 		w_1^+ < u_0 \leq v_0 \text{ and } w_0^+ < u_1 \leq v_1.
 	\]
 \end{lemma}

 \begin{proof}
 The category $\cC(\cZ)$ is Hom-finite and $\mocolim \mathbf{E}(Z) = 0$. Therefore, by Remark~\ref{R:subseq_induced_by_morphisms_being_zero_in_colimit} for every integer $m>0$ there exists an integer $m_Z > m$ such that, the composition
 	\begin{equation}\label{E:Uzero}
 	\begin{tikzcd}
          Z \arrow[r, "\varphi"] & E_{m} \arrow[r, "f_{m,m_Z}"] & E_{m_Z}
      \end{tikzcd}
 	\end{equation}
 	is zero.

    Fix indecomposable summands $\{v_0,v_1\}$ of $E_m$ and $\{w_0,w_1\}$ of $E_{m_Z}$ such that the composition
  \begin{center}
      \begin{tikzcd}
          \psi \colon \{v_0,v_1\} \arrow[r, "\iota"] & E_{m} \arrow[r, "f_{m,m_Z}"] & E_{m_Z} \arrow[r, "\pi"] & \{w_0,w_1\}
      \end{tikzcd}
  \end{center}
 	is non-zero and enabled by the ordering
 	\begin{equation}\label{E:psi}
 		v_0^- < w_0 < v_1^- < w_1.
 	\end{equation}
    Now fix an indecomposable summand $U$ of $Z$ and assume for a contradiction that 
     $U = \{u_0,u_1\}$ with
 	\begin{equation*}
 		w_1^+ < u_0 \leq v_0 \text{ and } w_0^+ < u_1 \leq v_1.
 	\end{equation*}
 	Combining these inequalities with \eqref{E:psi}, we obtain
 	\begin{equation*}
 		w_1^+ < u_0 \leq v_0 < w_0^+ < u_1 \leq v_1.
 	\end{equation*}
 	Therefore, there exists a non-zero morphism $\varphi' \colon U \rightarrow V$ enabled by the ordering
 	\begin{equation*}
 		u_0^- < v_0 < u_1^- < v_1.
 	\end{equation*}
 	Moreover, the composition $\psi \circ \varphi' \colon U \rightarrow W$ is non-zero because there are inequalities
 	\begin{equation*}
 		u_0 \leq v_0 \leq w_0 \text{ and } u_1 \leq v_1 \leq w_1.
 	\end{equation*}
 	Consequently, there is a non-zero morphism
    \begin{center}
 		\begin{tikzcd}
 			Z \arrow[r] & U \arrow[r,"\varphi'"] & V \arrow[r, "\iota"] & E_m \arrow[r, "f_{m,m_Z}"] & E_{m_Z}
 		\end{tikzcd}
 	\end{center}
 	which is a contradiction to the composition \eqref{E:Uzero} being zero for all morphisms $\varphi \colon Z \rightarrow E_m$.
 \end{proof}

\begin{lemma}\label{L:summands_look_like_one}
 	Assume $(\aisle,\coaisle)$ is right non-degenerate. Let $\mathbf{Y} = (Y_n,g_n)_{n>0}$ be a $(\Sigma^{-1}\coaisle$)-trivial sequence in $\coaisle$ that is compactly supported at $\Sigma^{-1}\coaisle$.
Then the following hold:
\begin{enumerate}
    \item For all integers $m > 0$, each non-trivial indecomposable direct summand of $Y_m$ is of the form $\arc{i}{j}{s}$ for some $\cZ$-indices $i,j \in Z(\rx)$ that lie in the same block $B \in \cP^c$ and an integer $s \geq -1$;
          
       \item For all integers $m \geq 0$ and $\delta \geq 3$, there exists an $m' > m$ such that the following holds: For every indecomposable summand $\arc{i}{j}{s}$ of $Y_m$ and every indecomposable summand $\arc{k}{\ell}{r}$ of $Y_{m'}$, consider the composition
        \begin{equation}\label{E:factoring}
        \begin{tikzcd}
            \psi \colon \arc{i}{j}{s} \arrow[r, "\iota"] & Y_{m} \arrow[r, "g_{m,m'}"] & Y_{m'} \arrow[r, "\pi"] & \arc{k}{\ell}{r},
        \end{tikzcd}
 	\end{equation}
 	where $\iota$ and $\pi$ denote the canonical inclusion and projection of indecomposable summands respectively. If $\psi$ is non-trivial and enabled by the ordering
 	\begin{equation*}
 	x_i^{--} < x_k^- < x_j^{(s-1)} < x_\ell^{(r)}
 \end{equation*}
       (up to a possible relabelling of $i$ and $j$ if $s = -1$), then $\ell$ is the $\cP^c$-predecessor of $i$ and $r \geq s+\delta$.
\end{enumerate}
 \end{lemma}
 
 \begin{proof}
	By Corollary~\ref{C:GZ_classification_coaisle} we have
 	\begin{equation*}
       \coaisle = \add \left( \bigcup_{B \in \cP^c} \cx \left( \bigcup_{q \in B} [x_q^{-},a_{q+1}) \right) \right), \;\;\; \Sigma^{-1} \coaisle = \add \left( \bigcup_{B \in \cP^c} \cx \left( \bigcup_{q \in B} [x_q,a_{q+1}) \right) \right).
 	\end{equation*}
  Thus an indecomposable object in $\coaisle$ that does not lie in $\Sigma^{-1}\coaisle$, in particular every summand of an entry of $\mathbf{Y}$, is of the form $\{x_i^-,x_j^{(s)}\}$ for some $i,j \in Z(\rx)$ in the same block $B$ of $\cP^c$ and some $s \geq -1$. 
   
Fix an integer $m > 0$ and let $\{x_i^-,x_j^{(s)}\}$ be a summand of the entry $Y_{m}$. Let $p$ be the $\cP^c$-predecessor of $i$ and fix an integer $\delta \geq 3$. As $s+\delta+1 \geq 3$, for all indices $q \in B$, the set $\{x_p^{(s+\delta+1)},x_q^+\}$ is a non-trivial arc of $\cZ$. Consider the object
\[ Z = \bigoplus_{q \in B}\{x_p^{(s+\delta+1)},x_q^+\} \in \Sigma^{-1}\coaisle.
    \] 
and set $m' = m_{Z}$ as defined in Lemma~\ref{L:factoring_crossings}. Let $\{x_k^-,x_\ell^{(r)}\}$  be a summand of $Y_{m'}$ such that the composition
  \begin{equation*}
        \begin{tikzcd}
            \psi \colon \arc{i}{j}{s} \arrow[r, "\iota"] & Y_{m} \arrow[r, "g_{m,m'}"] & Y_{m'} \arrow[r, "\pi"] & \arc{k}{\ell}{r}
        \end{tikzcd}
 	\end{equation*}
 (where $\iota$ and $\pi$ denote the canonical inclusion and projection of indecomposable summands respectively), is non-zero and enabled by the ordering
 \begin{equation}\label{E:rightorder}
 	x_i^{--} < x_k^- < x_j^{(s-1)} < x_\ell^{(r)}.
 \end{equation}
 Since $i$ and $j$, respectively $k$ and $\ell$, lie in the same block of $\cP^c$ which is non-crossing, all four of $i,j,k$ and $\ell$ lie in the same block $B$ of $\cP^c$. Moreover, by (\ref{E:rightorder}) we have $x_k \neq x_j^{(s)}$ and thus
  \begin{equation}\label{E:Z2-inequality}
    	x_k < x_k^+ \leq x_j^{(s)}.
  \end{equation}
The arc $\{x_p^{(s+\delta+1)},x_k^+\}$ is a non-trivial summand of $Z$. Therefore, by Lemma~\ref{L:factoring_crossings} and the inequalities (\ref{E:rightorder}) and (\ref{E:Z2-inequality}) we cannot have
    \[
    	x_\ell^{(r+1)} < x_p^{(s+\delta+1)} \leq x_i^-
    \]
    and so must have 
    \[
			 x_p^{(s+\delta+1)} \leq x_\ell^{(r+1)} \leq x_i^-.
    \]
    As $r+1 \geq 0$, this forces $p = \ell$ and $r \geq s+\delta$.
   \end{proof}

\begin{proposition} \label{P:summands_in_coaisle_are_neighbouring}
    Assume $(\aisle,\coaisle)$ is right non-degenerate. Let $\mathbf{Y} = (Y_n,g_{n})_{n > 0}$ be a minimal coaisle sequence with respect to $\coaisle$. Then there exists a subsequence $\mathbf{Y}_\mathbf{I} = (Y'_n,g'_n)_{n>0}$ of $\mathbf{Y}$ such that for all integers $m >0$ each non-trivial indecomposable direct summand of $Y'_m$ is of the form $\arc{i}{j}{s}$ for some $Z$-indices $i,j \in Z(\rx)$ where $j$ is the $\cP^c$-predecessor of $i$ and for some integer $s > m$.
\end{proposition}   

   \begin{proof}
    There exists a subsequence $\mathbf{Y_J}$ of $\mathbf{Y}$ such that for all $m,m' \in \mathbf{J}$ with $m<m'$ the cone of $g_{m,m'}$ lies in $\Sigma^{-2}\coaisle$ and if there exists a morphism as in (\ref{E:factoring}), then the conclusion of Lemma~\ref{L:summands_look_like_one}(ii) holds for $\delta=3$. Let $\mathbf{I}$ denote the subsequence of $\mathbf{J}$ given by removing its first entry and fix $m \in \mathbf{I}$. Choose a non-trivial indecomposable summand of $Y_m$ with canonical inclusion $\kappa$. By Lemma~\ref{L:summands_look_like_one} the summand is of the form $\arc{i}{j}{s}$. So, there exists a non-trivial morphism
\begin{equation*}
        \begin{tikzcd}
            \kappa \circ \varphi \colon \fromarc{x_i^-} \arrow[r, "\varphi"] & \arc{i}{j}{s} \arrow[r, "\kappa"] & Y_{m}.
        \end{tikzcd}
      \end{equation*}
    Let $m''$ denote the predecessor of $m$ in $\mathbf{J}$. By Theorem~\ref{T:GZ_classification_aisles}, the arc $\fromarc{x_i^-}$ lies in $\Sigma^{-1}\aisle$, and $\Hom(\Sigma^{-1}\aisle,\Sigma^{-2}\coaisle) =0$. So, the morphism $\kappa \circ \varphi$ factors through $g_{m'',m}$ and we obtain a commutative diagram
 	\begin{center}
 		\begin{tikzcd}[column sep = 4em]
 			& & \fromarc{x_i^-} \arrow[d,"\kappa \circ \varphi"] \arrow[ld, dashed] &
 			\\
 			\Sigma^{-1} \cone(g_{m'',m}) \arrow[r] & Y_{m''} \arrow[r, "g_{m'',m}"] & Y_{m} \arrow[r] & \cone(g_{m'',m}).
 		\end{tikzcd}
 	\end{center}
    Therefore there exists an indecomposable summand $\widetilde{Y}_{m''}$ of ${Y}_{m''}$ and a map $\varphi' \colon \fromarc{x_i^-} \to \widetilde{Y}_{m''}$ such that
    the composition
        \begin{center}
            \begin{tikzcd}[column sep = 4em]
            \psi \circ \varphi' \colon \fromarc{x_i^-} \arrow[r,"\varphi'"] & \widetilde{Y}_{m''} \arrow[r,"\iota"] & Y_{m''} \arrow[r, "g_{m'',m}"] & Y_{m} \arrow[r,"\pi"] & \arc{i}{j}{s}
            \end{tikzcd}
        \end{center}
 	is non-zero, where $\iota$ and $\pi$ denote the canonical inclusion and projection respectively. In particular, the morphism $\varphi' \colon \fromarc{x_i^-} \rightarrow \widetilde{Y}_{m''}$ is non-zero, and so, $x_i^-$ is an endpoint of $\widetilde{Y}_{m''}$. Moreover, since $\widetilde{Y}_{m''}$ lies in $\coaisle$, there exists a $\cZ$-index $q \in Z(\rx)$ and an integer $t\geq -1$ such that $\widetilde{Y}_{m''} = \arc{i}{q}{t}$. Therefore, there exists a non-zero morphism $\psi \colon \widetilde{Y}_{m''} = \arc{i}{q}{t} \rightarrow \arc{i}{j}{s}$, and so, the arcs $\Sigma \arc{i}{q}{t}$ and $\arc{i}{j}{s}$ cross. There is precisely one ordering of the endpoints that allows such a crossing, namely
  \begin{equation*}
      x_i^{--} < x_i^- < x_q^{(t-1)} < x_j^{(s)}.
  \end{equation*}
    Consequently, $\psi$ is enabled by this ordering and, by Lemma~\ref{L:summands_look_like_one}, we have that $s \geq t+3$ and $j$ is the $\cP^c$-predecessor of $i$. Hence, by induction over the sequence $\mathbf{I}$ we obtain $s > m$.
\end{proof}

\begin{proposition}\label{P:coaisle_seq_small_cones}
    Assume $(\aisle,\coaisle)$ is right non-degenerate. Let $\mathbf{Y} = (Y_n,g_n)_{n>0}$ be a minimal coaisle sequence with respect to $\coaisle$. Then there exists a subsequence $\mathbf{Y}_\mathbf{I}$ of $\mathbf{Y}$ such that, for all integers $m,m' \in \mathbf{I}$ with $m<m'$, the
        cone of $g_{m,m'}$ lies in
        \begin{equation*}
            \add \left( \bigcup_{j \in Z(\rx)} \cx \left((x_j^{(m-1)},a_{j+1})\right) \right).
        \end{equation*}
    In particular, $\mathbf{Y_I}$ is a Cauchy sequence with respect to $\cM_\coaisle$.
\end{proposition}

\begin{proof}
    As the sequence $\mathbf{Y}$ is a minimal coaisle sequence there exists a subsequence $\mathbf{Y}_\mathbf{I}$ which satisfies the conclusion of Proposition~\ref{P:summands_in_coaisle_are_neighbouring}. 
    Fix integers $m<m' \in \mathbf{I}$ and for each index $j \in [N]$ let $i_j \in [N]$ denote its $\cP^c$-successor. The objects $Y_{m'}$ and $\Sigma Y_m$ both lie in 
    \begin{equation*}
        \cC = \add \bigcup_{j \in Z(\rx)} \cx \left( \{x_{i_j}^{--},x_{i_j}^-\} \cup (x_j^{(m-1)},a_{j+1}) \right),
    \end{equation*}
   which is extension closed by Remark~\ref{R:convex_hull_ext_closed}. By the triangle
 	\begin{center}
 		\begin{tikzcd}
 			Y_{m} \arrow[r, "g_{m,m'}"] & Y_{m'} \arrow[r] & \cone(g_{m,m'}) \arrow[r] & \Sigma Y_{m},
 		\end{tikzcd}
 	\end{center}
    it follows that $\cone{g_{m,m'}} \in \cC$. Furthermore, $\mathbf{Y}$ stabilises at $\Sigma^{-2}\coaisle$. So, without loss of generality, by taking an appropriate subsequence if necessary, we may assume that $\cone{g_{m,m'}} \in \Sigma^{-2}\coaisle$. However, neither $x_{i_j}^{--}$ nor $x_{i_j}^-$ can be endpoints of an arc in $\Sigma^{-2}\coaisle$ by Corollary~\ref{C:GZ_classification_coaisle}. This shows the claim.
\end{proof}

The proof of Proposition~\ref{P:final_classification_coaisle_sequence} follows from Proposition~\ref{P:summands_in_coaisle_are_neighbouring} which proves Part~(i) and Proposition~\ref{P:coaisle_seq_small_cones} which proves Part~(ii).

\subsubsection{Bounding sequences by double fans} We now compute the module colimit of the sequence $\widehat{\mathbf{Y}}$ constructed in Proposition~\ref{P:reduce_to_RNDG_coaisle_sequences} in terms of double fans. Crucially, we show that such a sequence $\widehat{\mathbf{Y}}$ has a subsequence satisfying the following double fan-like properties.

\begin{definition}
    Given a t-structure $(\raisle,\rcoaisle)$ corresponding to a $\overline{\cZ}$-decorated non-crossing partition $(\cP_R,\rx_R = (\widetilde{x}_1, \ldots, \widetilde{x}_N))$ we may impose the following conditions with respect to $(\raisle,\rcoaisle)$ on a sequence $\mathbf{\widehat{Y}} = (\widehat{Y}_n,\widehat{g}_n)_{n>0}$ in $\cT$:
    \begin{enumerate} [label=$\mathbf{\widehat{Y}\arabic*}$]
        \item \label{P:overlineY1} for all $n>0$, the endpoints of $\widehat{Y}_n$ lie in the set
        \begin{equation*}
            S(\raisle) \cup \bigcup_{j \in Z(\rx_R)} (\widetilde{x}_j^{(n)}, a_{j+1});
        \end{equation*}
        \item \label{P:overlineY2} for all $n,n' \in I$ with $n' > n$ the
        cone of $g_{n,n'}$ lies in
        \begin{equation*}
            \add \left( \bigcup_{j \in Z(\rx_R)} \cx \left((\widetilde{x}_j^{(n-1)},a_{j+1})\right) \right).
        \end{equation*}
        
    \end{enumerate}
\end{definition}

\begin{lemma} \label{L:classification_of_tilde_seq}
    Let $\cM = \{B_t\}_{t\geq0}$ be a good metric and let $(\raisle,\rcoaisle)$ be a right non-degenerate t-structure on $\cC(\cZ)$. Let $\mathbf{\widehat{Y}} = (\widehat{Y}_n,\widehat{g}_n)_{n>0}$ be a sequence in $\cC(\cZ)$ with  $\raisle$-approximation $\mathbf{X}$ and $\rcoaisle$-approximation $\mathbf{Y}$ such that $\mathbf{X}$ stabilises at $0$ and $\mathbf{Y}$ is a minimal coaisle sequence with respect to $\rcoaisle$. Then there exists a $(\Sigma^{-1}\rcoaisle)$-trivial subsequence ${\mathbf{\widehat{Y}}}_\mathbf{I} = (\widehat{Y}'_{n},\widehat{g}'_{n})_{n>0}$ of $\mathbf{\widehat{Y}}$ which satisfies Properties~\ref{P:overlineY1} and \ref{P:overlineY2} with respect to $(\raisle,\rcoaisle)$.
\end{lemma}

\begin{proof}
    As $\mathbf{Y} = (Y_n,g_n)_{n>0}$ is a minimal coaisle sequence it is $(\Sigma^{-1}\rcoaisle)$-trivial. Thus, as $Y_n$ is the $\rcoaisle$-envelope of $\widehat{Y}_n$, it follows that $\widehat{\mathbf{Y}}$ is also $(\Sigma^{-1}\rcoaisle)$-trivial. Moreover, as the sequence $\mathbf{X} = (X_n,h_n)_{n>0}$ stabilises at $0$, by passing to a suitable truncation of $\widehat{\mathbf{Y}}$ if necessary, we may assume without loss of generality that every morphism $h_n$ is an isomorphism.

     $\mathbf{\widehat{Y}1}$: Since $(\raisle,\rcoaisle)$ is right non-degenerate and $\mathbf{Y}$ is a minimal coaisle sequence with respect to $\rcoaisle$, there exists a subsequence $\mathbf{Y_I}$ of $\mathbf{Y}$ that satisfies Proposition~\ref{P:final_classification_coaisle_sequence} with respect to the t-structure $(\raisle,\rcoaisle)$. Set $\widehat{\mathbf{Y}}_\mathbf{I}$ to be the corresponding subsequence of $\widehat{\mathbf{Y}}$. For each $n\in I$, we have the $(\raisle,\rcoaisle)$-approximation triangle
    \begin{center} \label{E:tilde_seq_triangle}
		\begin{tikzcd}
			X_n \arrow[r] & \widehat{Y}_{n} \arrow[r] & Y_{n} \arrow[r] & \Sigma X.
		\end{tikzcd}
	\end{center}
    Therefore, by \cite[Lemma~3.4]{Gratz--Zvonareva--2023--tstructures_in_clus_cat}, the endpoints of $\widehat{Y}_n$ lie in the union of the endpoints of $X_n$ and $Y_n$. As $X_n$ lies in $\raisle$, its endpoints lie in $S(\raisle)$. Let $(\cP_R,\rx_R = (\widetilde{x}_1, \ldots, \widetilde{x}_N))$ be the $\overline{\cZ}$-decorated non-crossing partition corresponding to $(\raisle,\rcoaisle)$. By Proposition~\ref{P:final_classification_coaisle_sequence}(i) every summand of $Y_n$ is of the form $\{\widetilde{x}_i^-,\widetilde{x}_j^{(s)}\}$ for some $s > n$. Note that the point $\widetilde{x}_i^-$ lies in $S(\raisle)$. Therefore, the endpoints of $Y_n$, and consequently the endpoints of $\widehat{Y}_n$, lie in
    \begin{equation*}
        S(\raisle) \cup \bigcup_{j \in Z(\rx_R)}(\widetilde{x}_j^{(n)},a_{j+1}).
    \end{equation*}

    $\mathbf{\widehat{Y}2}$: Take the sequence of approximation triangles $\Delta \widehat{\mathbf{Y}}_\mathbf{I} = (\Delta \widehat{Y}_n, \Delta \widehat{g}_n)_{n \in \mathbf{I}}$ of $\widehat{\mathbf{Y}}_\mathbf{I}$. The octrahedral axiom applied to the morphism of triangles $\Delta \widehat{g}_{n,n'}$ for $n' > n \gg 0$ implies that $\cone(\widehat{g}_{n,n'}) \cong \cone(g_{n,n'})$. So, the claim follows by Proposition~\ref{P:final_classification_coaisle_sequence}(ii).
\end{proof}

Now we prove that $\widehat{\mathbf{Y}}$ can be bounded by a double fan. To do so, we need the following lemma.

\begin{lemma} \label{L:mocolim_of_small_bumps}
    Let $\mathbf{E} = (E_n,f_n)_{n>0}$ be a sequence in $\cC(\cZ)$. For each index $\ell \in [N]$ choose a point $z_\ell \in \cZ$ that lies in $(a_\ell,a_{\ell+1})$. If for all integers $n>0$ the object $E_n$ lies in the subcategory
    \begin{equation*}
        \add \left(\bigcup_{\ell \in [N]} \cx \left((z_\ell^{(n-1)}, a_{\ell+1})\right)\right),
    \end{equation*}
    then we have $\mocolim\mathbf{E}=0$.
\end{lemma}

\begin{proof}
    Fix integers $m,n > 0$ and two distinct indices $k, \ell \in [N]$. Then there are no morphisms between arcs in $\cx((z_k^{(m-1)}, a_{k+1}))$ and $\cx((z_\ell^{(n-1)}, a_{\ell+1}))$. Therefore, we can decompose $\mathbf{E}$ as a direct sum of $N$ sequences, one for each index $\ell \in [N]$, and without loss of generality by considering one of these summands individually, we may assume the sequence $\mathbf{E}$ lies only in $\cx((z_\ell^{(n-1)}, a_{\ell+1}))$. 
    
    For a contradiction, assume that $W$ is an object in $\cC(\cZ)$ such that $\mocolim\mathbf{E}(W)$ is non-zero and without loss of generality assume $W = \{w_0,w_1\}$ is indecomposable. Then there exists an integer $n>0$ and a morphism $\varphi \colon W \rightarrow E_n$ such that for all integers $n'>n$ the composition
    \begin{center}
        \begin{tikzcd}
            W \arrow[r, "\varphi"] & E_n \arrow[r, "f_{n,n'}"] & E_{n'}
        \end{tikzcd}
    \end{center}
    is non-zero. Therefore, for all integers $n'>n$ there exists an indecomposable summand $V_{n'}$ of $E_{n'}$ with a non-trivial map from $W$. For each $n'>n$ choose points $v_0(n'),v_1(n') \in \cZ$ such that $V_{n'} = \{v_0(n'),v_1(n')\}$. Then both $v_0(n')$ and $v_1(n')$ lie in $(z_\ell^{(n'-1)},a_{\ell+1})$. Moreover, for all $n'>n$, the non-zero morphism $f_{n,n'} \circ \varphi$ is enabled by one of the following orderings:
        \[
            w_0^- < v_0(n') < w_1^- < v_1(n') \; \text{ or } \; w_0^- < v_1(n') < w_1^- < v_0(n').
        \]
    Either way, at least one of $w_0^-$ or $w_1^-$ lies in  $(z_\ell^{(n'-1)},a_{\ell+1})$ for infinitely many $n'>n$; a contradiction.
\end{proof}

\begin{lemma}  \label{lem:bound_tilde_seq_by_double_fan}
    Let $(\raisle,\rcoaisle)$ be a right non-degenerate t-structure on $\cC(\cZ)$. Let $\mathbf{\widehat{Y}} = (\widehat{Y}_n,\widehat{g}_n)_{n>0}$ be a $(\Sigma^{-1}\rcoaisle)$-trivial sequence in $\cC(\cZ)$ that satisfies Properties~\ref{P:overlineY1} and \ref{P:overlineY2} with respect to the t-structure $(\raisle,\rcoaisle)$.    
    Then there exists a double fan $\mathbf{F}$ that is a Cauchy sequence with respect to the coaisle metric $\cM_\rcoaisle$ and an isomorphism
    \begin{equation*}
        \mocolim\mathbf{\widehat{Y}} \cong \mocolim\mathbf{F}.
    \end{equation*}
\end{lemma}

\begin{proof}
    Let $(\cP_R,\rx_R = (\widetilde{x}_1, \ldots, \widetilde{x}_N))$ be the $\overline{\cZ}$-decorated non-crossing partition associated to $(\raisle,\rcoaisle)$. Since the endpoints of $\widehat{Y}_1$ lie in the set
        \begin{equation*}
            S(\raisle) \cup \bigcup_{\ell \in Z(\rx_R)} (\widetilde{x}_\ell^{+}, a_{\ell+1});
        \end{equation*}
    its indecomposable summands can be grouped into three types:
    \begin{enumerate} [label=\textbf{T\arabic*}]
        \item Arcs that have both endpoints in $S(\raisle)$.
        \item Arcs of the form $\{v, \widetilde{x}_\ell^{(r)}\}$ for a point $v \in S(\raisle)$, an index $\ell \in Z(\rx_R)$ and an integer $r > 1$.
        \item Arcs of the form $\{\widetilde{x}_i^{(t)}, \widetilde{x}_j^{(s)}\}$ for indices $i,j \in Z(\rx_R)$ and integers $s,t > 1$. 
    \end{enumerate}
    Given an object $X$ of $\cC(\cZ)$ we denote by $\mathrm{ind}(X)$ its set of indecomposable summands. We define the double fan $\mathbf{F} = (F_n,\zeta_n)_{n>0}$ as a direct sum of indecomposable double fans $\bigoplus_{Z \in \mathrm{ind}(\widehat{Y}_1)} \mathbf{F(Z)}$ where $\mathbf{F(Z)} = (F(Z)_n,\zeta(Z)_n)_{n>0}$ depends on the type of $Z$: 
    \begin{enumerate} [label=\textbf{T\arabic*}]
    \item Define $\mathbf{F(Z)}$ to be the constant double fan $(Z,1_Z)_{n>0}$. 
    \item If $Z= \{v,\widetilde{x}_\ell^{(r)}\}$, then define $F(Z)_n = \{v, \widetilde{x}_\ell^{(n+1)}\}$ with canonical morphism $\zeta(Z)_n$. 
    \item If $Z = \{\widetilde{x}_i^{(t)}, \widetilde{x}_j^{(s)}\}$, then define $F(Z)_n = \{\widetilde{x}_i^{(n+1)}, \widetilde{x}_j^{(n+1)}\}$ with canonical morphism $\zeta(Z)_n$.
    Note that $i \neq j$ because no summand of $\widehat{Y}_1$ lies in $\Sigma^{-1}\rcoaisle$ and both $s,t > 1$. Therefore, $F(Z)_n$ is an arc and $F(Z)$ is well-defined.
    \end{enumerate}
   For all integers $n'>n>0$, the cone of $\zeta_{n,n'}$ is isomorphic to $\bigoplus_{Z \in \mathrm{ind}(\widehat{Y}_1)} \cone \zeta(Z)_{n,n'}$. Consequently, each of its indecomposable summands is of the form $\{\widetilde{x}_k^{(n'+1)}, \widetilde{x}_k^{(n)}\}$ for some index $k \in Z(\rx_R)$. In particular, the double fan $\mathbf{F}$ is a Cauchy sequence with respect to the coaisle metric $\cM_\rcoaisle$.

    We claim that, for all integers $n>0$, there exists a morphism $\Phi_n \colon F_n \rightarrow \widehat{Y}_n$ that induces a morphism of triangles
    \begin{center}
 		\begin{tikzcd}
 			F_n \arrow[r, "\Phi_n"] \arrow[d, "\zeta_{n}"] & \widehat{Y}_n \arrow[d, "\widehat{g}_{n}"] \arrow[r] & \cone{\Phi_n} \arrow[d, "\xi_n"] \arrow[r] & \Sigma F_n \arrow[d]
 			\\
 			F_{n+1} \arrow[r, "\Phi_{n+1}"] & \widehat{Y}_{n+1} \arrow[r] & \cone{\Phi_{n+1}} \arrow[r] & \Sigma F_{n+1}
 		\end{tikzcd}
 	\end{center}
    and moreover, that 
    \begin{equation*}
        \cone{\Phi_n} \in \add \left( \bigcup_{\ell \in Z(\rx_R)} \cx \left((\widetilde{x}_\ell^{(n-1)}, a_{\ell+1})\right)\right).
    \end{equation*}
    This is sufficient to prove the main statement: Indeed, if $\mathbf{C} = (\cone{\Phi_n},\xi_n)_{n>0}$, then by Lemma~\ref{L:mocolim_of_small_bumps} we have $\mocolim \mathbf{C} = 0$, and, as the Yoneda embedding is homological, we have $\mocolim \mathbf{F} \cong \mocolim \mathbf{\widehat{Y}}$.
    
    We inductively construct $\Phi_n$. First consider the case $n=1$. 
    We define $\Phi_1 = \bigoplus_{Z \in \mathrm{ind}(\widehat{Y}_1)}\varphi_Z$, where $\varphi_Z \colon F_1(Z) \to Z$ depends on the type of $Z$:
    \begin{enumerate} [label = \textbf{T\arabic*}]
    \item Set $\varphi_Z = \id_Z$, and thus, $\cone \varphi_Z = 0$.
    \item If $Z= \{v,\widetilde{x}_\ell^{(r)}\}$, then $F(Z)_1 = \{v, \widetilde{x}_\ell^{(2)}\}$. Set $\varphi_Z$ to be the canonical morphism enabled by the ordering
    \begin{equation*}
        v^- < v < \widetilde{x}_\ell^+ < \widetilde{x}_\ell^{(r)}.
    \end{equation*} 
    In this case $\cone \varphi_Z \cong \{\widetilde{x}_\ell^+, \widetilde{x}_\ell^{(r)}\}$.
    \item If $Z = \{\widetilde{x}_i^{(t)}, \widetilde{x}_j^{(s)}\}$ then $F(Z)_1 = \{\widetilde{x}_i^{(2)}, \widetilde{x}_j^{(2)}\}$ and $i \neq j$. Set $\varphi_Z$ to be the canonical morphism enabled by the ordering
    \begin{equation*}
        \widetilde{x}_i^+ < \widetilde{x}_i^{(t)} < \widetilde{x}_j^+ < \widetilde{x}_j^{(s)}.
    \end{equation*}
    In this case $\cone(\varphi_Z) \cong \{\widetilde{x}_i^+, \widetilde{x}_i^{(t)}\} \oplus \{\widetilde{x}_j^+, \widetilde{x}_j^{(s)}\}$. 
    \end{enumerate}
    Consequently $\cone(\Phi_1) \cong \bigoplus_{Z \in \widehat{Y}_1} \cone(\varphi_Z)$ is isomorphic to a direct sum of arcs of the form $\{\widetilde{x}_k^+, \widetilde{x}_k^{(r)}\}$ such that $r>1$ and $k \in Z(\rx_R)$, and so it lies in
    \begin{equation*}
        \add \left( \bigcup_{\ell \in Z(\rx_R)} \cx \left((\widetilde{x}_\ell, a_{\ell+1})\right) \right).
    \end{equation*}

    Now suppose that for some integer $n>0$ there exists a morphism $\Phi_n \colon F_n \rightarrow \widehat{Y}_n$ with the required properties. Each indecomposable summand of $\Sigma^{-1}\cone\zeta_n$ is of the form $\fromarc{\widetilde{x}_k^{(n+1)}} = \{\widetilde{x}_k^{(n+1)}, \widetilde{x}_k^{(n+3)}\}$ for some index $k \in Z(\rx_R)$. Recall the endpoints of $\widehat{Y}_{n+1}$ lie in 
    \begin{equation*}
        S(\raisle) \cup \bigcup_{\ell \in Z(\rx_R)} (\widetilde{x}_{\ell}^{(n+1)}, a_{\ell+1}),
    \end{equation*}
    and because $n>0$ we have $\widetilde{x}_k^{(n+1)} \notin S(\raisle)$. Therefore, there are no summands of $\widehat{Y}_{n+1}$ that have the point $\widetilde{x}_k^{(n+1)}$ as an endpoint and thus no non-trivial morphisms $\Sigma^{-1}\cone\zeta_n \rightarrow \widehat{Y}_{n+1}$. So, we obtain a commutative diagram
    \begin{center}
 	\begin{tikzcd}
            \Sigma^{-1}\cone(\zeta_n) \arrow[r] & F_n \arrow[r, "\zeta_n"] \arrow[d,"\Phi_n"] & F_{n+1} \arrow[r] \arrow[d,"\Phi_{n+1}",dashed] & \cone(\zeta_n).
 		\\
 		& \widehat{Y}_n \ar[r, "\widehat{g}_n"] & \widehat{Y}_{n+1} &
 	\end{tikzcd}
    \end{center}
    Then by the ($3 \times 3$)-axiom (cf.\ \cite[Lemma~2.6]{May-triangulated}), there exists an object $C$ and a diagram
 	\begin{center}
  \begin{equation}\label{E:3x3_fan_approx}
 		\begin{tikzcd}
 			F_n \arrow[r, "\Phi_n"] \arrow[d, "\zeta_n"] & \widehat{Y}_n \arrow[d, "\widehat{g}_n"] \arrow[r] & \cone{\Phi_n} \arrow[d] \arrow[r] & \Sigma F_n \arrow[d]
 			\\
 			F_{n+1} \arrow[r, "\Phi_{n+1}"] \arrow[d] & \widehat{Y}_{n+1} \arrow[r] \arrow[d] & \cone{\Phi_{n+1}} \arrow[r] \arrow[d] & \Sigma F_{n+1} \arrow[d]
 			\\
 			\cone{\zeta_n} \arrow[r] \arrow[d] & \cone{\widehat{g}_n} \arrow[r] \arrow[d] & C \arrow[r] \arrow[d] & \Sigma \cone{\zeta_n} \arrow[d]
 			\\
 			\Sigma F_n \arrow[r, "\Sigma \Phi_n"] & \Sigma \widehat{Y}_n \arrow[r] & \Sigma \cone{\Phi_n} \arrow[r] & \Sigma^2 F_n
 		\end{tikzcd}
   \end{equation}
 	\end{center}
 	that is commutative everywhere except in the bottom right square, which is anticommutative, and whose rows and columns are distinguished triangles in $\cC(\cZ)$. Note that the subcategory
  \[
    \add \left(\bigcup_{\ell \in Z(\rx_R)} \cx ((\widetilde{x}_\ell^{(n-2)}, a_{\ell+1}))\right)
    \]
    is extension closed by \cite[Lemma~3.4]{Gratz--Zvonareva--2023--tstructures_in_clus_cat} and Equation (\ref{E:Ext1}). By the inductive assumption it contains $\cone(\Phi_n)$. 
    Since it also contains $\Sigma \cone(\zeta_n)$ and $\cone \widehat{g}_n$, by row three in (\ref{E:3x3_fan_approx}) it also contains $C$. Therefore by column three in (\ref{E:3x3_fan_approx}) it contains $\cone(\Phi_{n+1})$. Consider now the second row in (\ref{E:3x3_fan_approx}). The support of $\add \widehat{Y}_{n+1}$ is a subset of
 \[
    S(\raisle) \cup \bigcup_{\ell \in Z(\rx_R)} (\widetilde{x}_\ell^{(n+1)}, a_{\ell+1})
\]
  and the support of $\add \Sigma F_{n+1}$ is a subset of
  \[
    S(\Sigma\raisle) \cup \bigcup_{\ell \in Z(\rx_R)} (\widetilde{x}_\ell^{(n)}, a_{\ell+1}).
\]
  Therefore, the support of $\add \cone(\Phi_{n+1})$ is contained in the intersection
  \begin{equation}\label{E:intersection}
    \left( S(\raisle) \cup \bigcup_{\ell \in Z(\rx_R)} (\widetilde{x}_\ell^{(n)}, a_{\ell+1}) \right) \cap  \bigcup_{\ell \in Z(\rx_R)} (\widetilde{x}_\ell^{(n-2)}, a_{\ell+1}) = \bigcup_{\ell \in Z(\rx_R)} (\widetilde{x}_\ell^{(n)}, a_{\ell+1}),
  \end{equation}
   where equality holds because for all $n>1$ and $\ell \in Z(\rx_R)$ we have
    $S(\raisle) \cap  (\widetilde{x}_\ell^{(n-2)}, a_{\ell+1}) = \varnothing$. Therefore the morphism $\Phi_{n+1}$ has the desired properties.
\end{proof}

\begin{proposition} \label{P:metric_lies_in_comb}
    The metric completion $\fS_{\cM_\coaisle}$ of $\cC(\cZ)$ with respect to the coaisle metric $\cM_\coaisle$ is a subcategory of the combinatorial completion $\fR_{(\cP,\rx)}$ of $\cC(\cZ)$ with respect to the $\overline{\cZ}$-decorated non-crossing partition $(\cP,\rx)$.
\end{proposition}

\begin{proof}
    Let $E$ be an object in the metric completion $\fS_{\cM_\coaisle}$. Then $E$ is compactly supported at some $t \geq 0$. Let $\widehat{\aisle}$ be a largest aisle in $\coaisle$ with t-structure $(\widehat{\aisle},\widehat{\coaisle})$ and let $\rcoaisle = \coaisle \cap \Sigma^{2(t+1)}\widehat{\coaisle}$. By Proposition~\ref{P:reduce_to_RNDG_coaisle_sequences} and Lemma~\ref{L:classification_of_tilde_seq}, there exists a $(\Sigma^{-t}\rcoaisle)$-trivial sequence $\mathbf{\widehat{Y}}$ in $\Sigma^t\widehat{\coaisle}$ satisfying \ref{P:overlineY1} and \ref{P:overlineY2} with respect to the t-structure $(\Sigma^{-(t-1)}\raisle,\Sigma^{-(t-1)}\rcoaisle)$ such that
    \begin{equation*}
        E \cong \mocolim\mathbf{\widehat{Y}}.
    \end{equation*}
    Therefore, by Lemma~\ref{lem:bound_tilde_seq_by_double_fan} with t-structure $(\Sigma^{-(t-1)}\raisle,\Sigma^{-(t-1)}\rcoaisle)$, there exists a double fan $\mathbf{F}$ which is a Cauchy sequence with respect to $\cM_{\Sigma^{-(t-1)}\rcoaisle}$, and consequently also with respect to $\cM_{\rcoaisle}$, and isomorphisms
    \begin{equation*}
        E \cong \mocolim\mathbf{E} \cong \mocolim\mathbf{\widehat{Y}} \cong \mocolim\mathbf{F}.
    \end{equation*}
    As $E$ is compactly supported at $t$ with respect to $\cM_\coaisle$, by the existence of these isomorphisms, the double fan $\mathbf{F}$ is also compactly supported at $t$ with respect to $\cM_\coaisle$. Moreover, as $\rcoaisle \subseteq \coaisle$, the double fan $\mathbf{F}$ is also a Cauchy sequence with respect to $\cM_\coaisle$. The double fan $\mathbf{F}$ decomposes as a direct sum $\bigoplus_{i=1}^n \mathbf{F}_i$ of indecomposable double fans which are also compactly supported and Cauchy, and by Proposition~\ref{P:comb_comp_equals_double_fan_mocolims} their module colimits lie in the combinatorial completion $\fR_{(\cP,\rx)}$. Hence so does $E \cong \bigoplus_{i=1}^n \mocolim \mathbf{F}_i$.
\end{proof}

Now Theorem~\ref{T:completions agree} follows by Corollary~\ref{C:comb_lies_in_metric} and Proposition~\ref{P:metric_lies_in_comb}.

\subsection{A combinatorial description of the coaisle completions}

The interpretation of the metric completion $\mathfrak{S}_{\cM_\coaisle}$ of $\cC(\cZ)$ with respect to a coaisle metric in terms of its combinatorial completion allows for a holistic combinatorial description of $\mathfrak{S}_{\cM_\coaisle}$.

\begin{theorem}\label{T:combinatorial description of completion}
    Let $(\aisle,\coaisle)$ be a t-structure on $\cC(\cZ)$ corresponding to a $\overline{\cZ}$-decorated non-crossing partition $(\cP,\rx)$. The metric completion $\mathfrak{S}_{\cM_\coaisle}$ of $\cC(\cZ)$ with respect to the coaisle metric $\cM_\coaisle$ is equivalent to
    \[
        \mathfrak{S}_{\cM_\coaisle}(\cC(\cZ)) \simeq \bigoplus_{B \in \cP} \mathfrak{S}_{\cM_\coaisle}(B),
    \]
    where each $\mathfrak{S}_{\cM_\coaisle}(B)$ can be described combinatorially as follows:
    \begin{itemize}
        \item The indecomposable objects of $\mathfrak{S}_{\cM_\coaisle}(B)$ are in one-to-one correspondence with arcs of $\overline{\cx}(\overline{\cZ}_B)$.
        \item Let $F$ and $G$ be arcs of $\overline{\cx}(\overline{\cZ}_B)$. Then we have 
        \begin{equation*}
		\Hom_{\mathfrak{S}_{\cM_\coaisle}(B)}(F, G) \cong \begin{cases}
			\bK, &\text{if } F = \{f, f'\} \text{ and }  G = \{g,g'\} \text{ with } \\
    & f \leq g < {f'}^- \text{ and } f' \leq  g' < f^-
			\\
			0, &\text{ otherwise.}\\
		\end{cases}
	\end{equation*}
 \item{Let $F = \{f,f'\}$, $G = \{g,g'\}$ and $H = \{h,h'\}$ be arcs of $\overline{\cx}(\overline{\cZ}_B)$ with non-trivial maps $\varphi \colon F \to G $ and $\psi \colon G \to H$ enabled by the inequalities
\[
    f \leq g < {f'}^-, \;\; f' \leq g' < f^- \text{ and } g \leq h < {g'}^-, \;\; g' \leq h' < g^-.
\]
Then $\psi \circ \varphi \neq 0$ if and only if 
\[
    f \leq g \leq h \text{ and } f' \leq g' \leq h'.
\]}
    \end{itemize}
Moreover, $\mathfrak{S}_{\cM_\coaisle}$ is a triangulated category with suspension functor $\Sigma$ acting on arcs by a one-step clockwise rotation, i.e.\ given an arc $\{z_0,z_1\}$ we have $\Sigma \{z_0,z_1\} = \{z_0^-,z_1^-\}$, where for an accumulation point $a \in L(\cZ)$ we write $a^- = a = a^+$. Its distinguished triangles are diagrams which are isomorphic to colimits in $\Mod \cC(\cZ)$ of compactly supported Cauchy sequences of triangles in $\cC(\cZ)$. In particular, whenever $F = \{f,f'\}$ and $G = \{g,g'\}$ are indecomposable with a non-trivial map $\varphi \colon F \to G$ enabled by the inequalities
\[
    f \leq g < {f'}^- \text{ and } f' \leq g' < f^-
\]
we obtain a triangle
\[
    F \xrightarrow{\varphi} G \to \{f^-,g\} \oplus \{{f'}^-,g'\} \to \Sigma F,
\]
where we set $\{b,b'\} = 0$ if $b' \in \{b^-,b,b^+\}$.
\end{theorem}

\begin{proof}
    This is an immediate consequence of Theorem~\ref{T:completions agree}, Lemma~\ref{L:morphisms_between_double_fans} and \cite[Theorem~15]{neeman2020metricssurvey}.
\end{proof}

\begin{remark}
    In upcoming work, we prove there is a similar combinatorial description for the completion of $\cC(\cZ)$ with respect to any good metric $\cM = \{B_t\}_{t \geq 0}$ such that for all $t \geq 0$ we have $B_t = \add B_t$.
\end{remark}

Let $\overline{\cC(\cZ)}$ denote the combinatorial completion of $\cC(\cZ)$ as defined in \cite{Paquette--Yildirim--2021:combinatorial_completion}.

\begin{corollary}\label{C:PY-equivalence}
    Let $(\aisle,\coaisle)$ be a non-degenerate t-structure with corresponding $\cZ$-decorated non-crossing partition $(\cP_0, \rx)$, where $\cP_0 = \{[N]\}$ is the coarsest partition of $[N]$, and $\rx \in \cZ^N$. Then there is an equivalence of categories
    \[
        \mathfrak{S}_{\cM_\coaisle} \simeq \overline{\cC(\cZ)}.
    \]
\end{corollary}

\begin{remark}
    By Theorem~\ref{T:combinatorial description of completion} we see that the equivalence of categories $ \mathfrak{S}_{\cM_\coaisle} \simeq \overline{\cC(\cZ)}$ in Corollary~\ref{C:PY-equivalence} commutes with suspension, and preserves triangles of the form $A \to B \to C \to \Sigma A$ where $A$ and $B$ are indecomposable. It is entirely plausible that the equivalence of categories is in fact a triangulated equivalence. This remains to be investigated.
\end{remark}

More generally, in Theorem~\ref{T:combinatorial description of completion}, each of the direct summands $\mathfrak{S}_{\cM_\coaisle}(B)$ could, as a category, be constructed in an analogous manner to the combinatorial completion in \cite{Paquette--Yildirim--2021:combinatorial_completion}. The only very minor adaptation would be to not necessarily replace all the accumulation points by lines of integers, but instead replace an appropriate subset of the accumulation points (cf.\ start of \cite[Section~3]{Paquette--Yildirim--2021:combinatorial_completion}), and then take the relevant Verdier quotient with respect to the newly generated intervals.

\section{An example} \label{S:example}

    In this section we explicitly compute the completion of a discrete cluster category with $N=10$ accumulation points with respect to a specific coaisle metric.

    \subsection{T-structures} \label{SS:EX:tstructures}
    Consider the non-crossing partition
    \begin{equation*}
        \cP = \{ \{1\}, \{2,3,9\}, \{4\}, \{5,6\}, \{7\}, \{8\}, \{10\}\}
    \end{equation*}
    of $[10]$. Let $\rx$ be the $\overline{\cZ}$-decoration of $\cP$ given by $\rx = (x_1,a_3,x_3,a_4,a_6,x_6,a_7,a_8,x_9,a_{10})$ where for the indices $i=1,3,6,9$ we have $x_i \in \cZ$ (cf. Definition~\ref{D:decoratednc}). The pair $(\cP,\rx)$ defines an aisle $\aisle$ of $\cC(\cZ)$ (cf.\ Theorem~\ref{T:GZ_classification_aisles}). This aisle is illustrated by the left-hand diagram in Figure~\ref{F:tstructure}, where the indecomposable objects of $\aisle$ correspond to the arcs that lie in the shaded region. The corresponding coaisle $\coaisle$ of $\cC(\cZ)$ is illustrated by the right-hand diagram in Figure~\ref{F:tstructure} (cf.\ Corollary~\ref{C:GZ_classification_coaisle}). Note the Kreweras complement of $\cP$ is
    \begin{equation*}
        \cP^c = \{\{1,9,10\}, \{2\}, \{3,4,6,7,8\}, \{5\}\}.
    \end{equation*}

    \begin{figure}[H]
        \centering
        \includegraphics[scale=0.8]{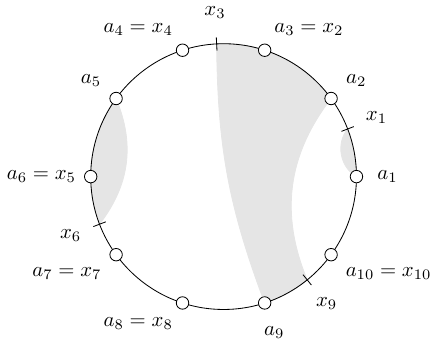} \;\;\; 
        \includegraphics[scale=0.8]{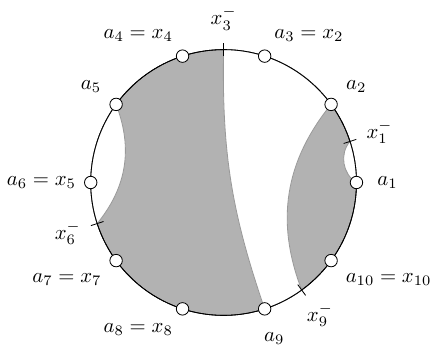}
        \caption{The aisle $\aisle$ (left) and the coaisle $\coaisle$ (right) of the t-structure $(\aisle,\coaisle)$ of $\cC(\cZ)$.}
        \label{F:tstructure}
    \end{figure}

    \subsection{Right non-degenerate t-structures} \label{SS:EX:RNDGcoaisles} The t-structure $(\aisle,\coaisle)$ is not right non-degenerate (cf.\ Definition~\ref{D:RNDG}). Indeed, the non-empty intersection $\bigcap_{n \in \bZ} \Sigma^n\coaisle$ is illustrated by the striped region of the left-hand diagram in Figure~\ref{F:RNDGcoaisles}. Notice the support of this intersection is the union of those segments $(a_i,a_{i+1})$ that lie in $S(\coaisle)$; that is those segments with $x_i = a_i$. We may define a new $\overline{\cZ}$-decoration of the partition $\cP$ to obtain a right non-degenerate t-structure. In particular, fix a decoration $\widetilde{\rx} = (\widetilde{x}_1,a_3,\widetilde{x}_3,\widetilde{x}_4,a_6,\widetilde{x}_6,\widetilde{x}_7,\widetilde{x}_8,\widetilde{x}_9,\widetilde{x}_{10})$ where for indices $i \in [N] \setminus \{2,5\}$ we have $\widetilde{x}_i \in \cZ$. Then the pair $(\cP,\widetilde{\rx})$ defines a right non-degenerate t-structure $(\widetilde{\aisle}, \widetilde{\coaisle})$ whose coaisle $\widetilde{\coaisle}$ is illustrated by the right-hand diagram in Figure~\ref{F:RNDGcoaisles}. Notice that for all indices $i \in [N]$ we have $\widetilde{x}_i \neq a_i$.

    \begin{figure}[H]
        \centering
        \includegraphics[scale=0.8]{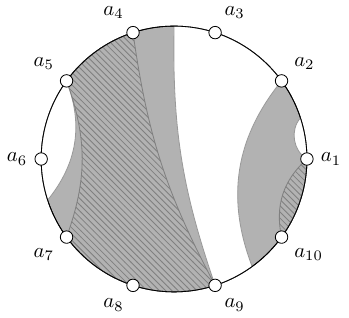}  
        \;\;\;
        \includegraphics[scale=0.8]{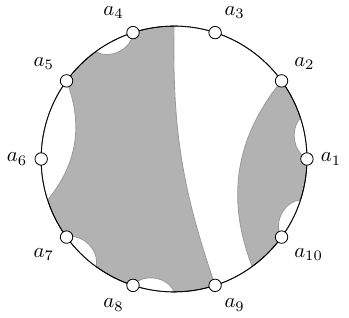} 
        \caption{A coaisle $\coaisle$ of $\cC(\cZ)$ that is not right-non degenerate (left) and a coaisle $\widetilde{\coaisle}$ of $\cC(\cZ)$ that is right non-degenerate (right). The striped region represents the intersection $\bigcap_{n \in \bZ} \Sigma^n\coaisle$.}
        \label{F:RNDGcoaisles}
    \end{figure}

    \subsection{Largest aisles contained in coaisles.} \label{SS:EX:largestaisles}
    The right non-degenerate coaisle $\widetilde{\coaisle}$ illustrated by the right-hand diagram in Figure~\ref{F:RNDGcoaisles} is very similar to the coaisle $\coaisle$ illustrated by the right-hand diagram in Figure~\ref{F:tstructure}. In particular, up to equivalence, $\widetilde{\coaisle}$ is obtained from $\coaisle$ by removing the aisle
    \begin{equation*}
        \add \left( \bigcup_{i \in \{4,7,8,10\}} \cx \left((a_i,\widetilde{x}_i\right]) \right).
    \end{equation*}
    However, $\coaisle$ contains a larger aisle; namely the aisle $\widehat{\aisle}$ illustrated by the striped region in the left-hand diagram in Figure~\ref{F:largest_aisles}. This aisle is a representative of the unique equivalence class of largest aisles contained in $\coaisle$. The intersection $\rcoaisle = \coaisle \cap \widehat{\coaisle}$ is a right non-degenerate coaisle as depicted by the right-hand diagram in Figure~\ref{F:largest_aisles}.

    \begin{figure}[H]
        \centering
        \includegraphics[scale=0.8]{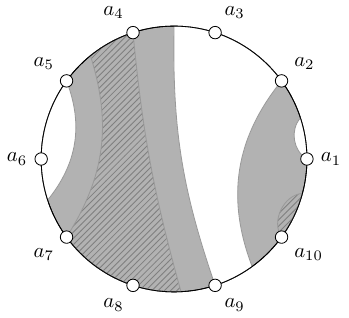}
        \;\;\;
        \includegraphics[scale=0.8]{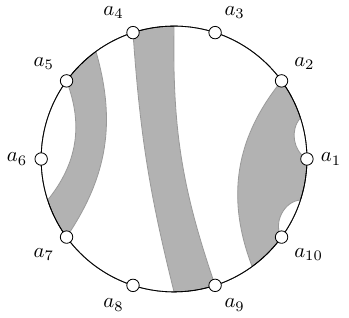}
        \caption{A largest aisle $\widehat{\aisle}$ contained in the coaisle $\coaisle$ (left) and the induced right non-degenerate coaisle $\rcoaisle = \coaisle \cap \widehat{\coaisle}$ (right).}
        \label{F:largest_aisles}
    \end{figure}

    \subsection{Double fans.} \label{SS:EX:doublefans}
    To compute the metric completion of $\cC(\cZ)$ with respect to the coaisle metric $\cM_\coaisle$ it suffices to restrict our attention to compactly supported Cauchy double fan sequences (cf.\ Theorem~\ref{T:completions agree} and Proposition~\ref{P:comb_comp_equals_double_fan_mocolims}). The left-hand diagram in Figure~\ref{F:doublefanseq} depicts four examples of double fan sequences (cf.\ Definition~\ref{D:double_fan}). The entries of the sequences are illustrated as solid and then dotted lines. The module colimits of the double fan sequences are depicted by the bold lines. We denote the module colimits of $\mathbf{F_1}$, $\mathbf{F_2}$, $\mathbf{F_3}$ and $\mathbf{F_4}$ by $\{a_2,a_9\}$, $\{z_2,a_3\}$, $\{z_3,a_6\}$ and $\{z_4,z'_4\}$ respectively (cf. Definition~\ref{D:double_fan}). The sequence $\mathbf{E}$ depicted in the right-hand diagram in Figure~\ref{F:doublefanseq} is not a double fan sequence because the accumulation point $a_2$ is the limit of both sequences of endpoints. In particular, the module colimit of the sequence $\mathbf{E}$ is trivial (cf.\ Lemma~\ref{L:mocolim_of_small_bumps}).

    \begin{figure}[H]
        \centering
        \includegraphics[scale=0.8]{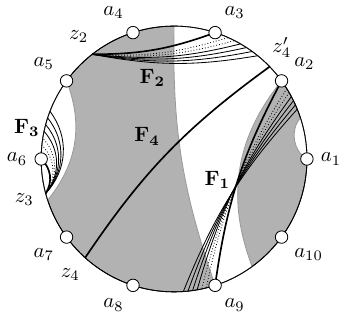}
        \;\;\;
        \includegraphics[scale=0.8]{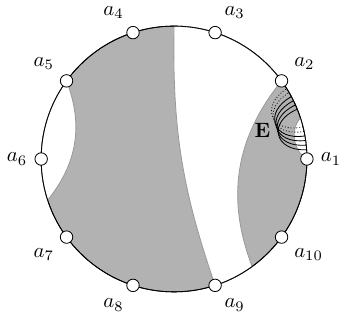}
        \caption{Four examples of double fan sequences (left) and a non-example of a double fan sequence (right).}
        \label{F:doublefanseq}
    \end{figure}

    \subsection{Morphisms between module colimits of double fans.} \label{SS:EX:morphismsbetweendoublefans}
    Morphisms in $\Mod\cC(\cZ)$ between module colimits of double fan sequences can be interpreted from the combinatorial model (cf.\ Lemma~\ref{L:morphisms_between_double_fans}). In Figure~\ref{F:mocolimmorphism} there are three arcs of $\overline{\cZ}$, namely $Z = \{z,z'\}$, $W = \{a_2,a_9\}$ and $V = \{a_2,v\}$ all of which can be realised as the module colimit of an appropriate double fan sequence (cf.\ Remark~\ref{R:every arc is a mocolim}).  There exists a morphism $Z \rightarrow V$ and a morphism $V \rightarrow Z$ enabled by the orderings
    \begin{equation*}
        z \leq v < {z'}^- \text{ and } z' \leq a_2 < z^-,
    \end{equation*}
    respectively
    \begin{equation*}
        a_2 \leq z < v^- \text{ and } v \leq z' < a_2^-=a_2.
    \end{equation*}
    There is also a morphism $V \rightarrow W$ enabled by the ordering
    \begin{equation*}
        v \leq a_9 < a_2^-=a_2 \text{ and } a_2 \leq a_2 < v^-.
    \end{equation*}
    However, there is no morphism $W \rightarrow V$ because neither $v$ nor $a_2$ lies in the interval $[a_9,a_2^-)$. Similarly, there exist morphisms $W \rightarrow Z$ and $Z \rightarrow W$. Moreover, we have
    \begin{equation*}
        z \leq v \leq a_9 \text{ and } z' \leq a_2 \leq a_2,
    \end{equation*}
    and so, the composition $Z \rightarrow V \rightarrow W$ is non-zero.

    \begin{figure}[H]
        \centering
        \includegraphics[scale=0.8]{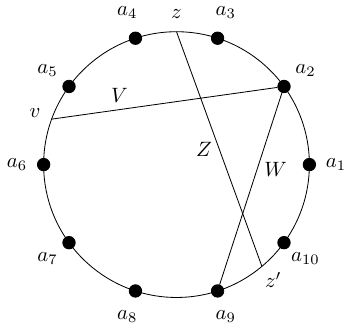}
        \caption{Three arcs $V$, $W$ and $Z$ of $\overline{\cZ}$.}
        \label{F:mocolimmorphism}
    \end{figure}

    \subsection{Cauchy and compactly supported double fans.} \label{SS:EX:cauchyandcptlysuppdoublefans}
    Whether or not a double fan sequence is Cauchy and/or compactly supported can easily be read from the combinatorial model (cf.\ Lemmata~\ref{L:double_fan_cauchy} and \ref{L:double_fan_cptly_supp}). The sequences $\mathbf{F_1}$ and $\mathbf{F_4}$ in the left-hand diagram of Figure~\ref{F:doublefanseq} are both Cauchy because both $x_1 \neq a_2$ and $x_8 \neq a_9$, and $\mathbf{F_4}$ is a constant sequence, whereas the sequences $\mathbf{F_2}$ and $\mathbf{F_3}$ are not Cauchy as $x_2=a_2$ and $x_5=a_5$. Moreover, the sequence $\mathbf{F_1}$ is compactly supported, because $2$ and $9$ lie in the same block of $\cP$, $x_2 \neq a_3$ and $x_9 \neq a_{10}$, and similarly we see that $\mathbf{F_3}$ is also compactly supported. However, the sequences $\mathbf{F_2}$ and $\mathbf{F_4}$ are not compactly supported. Indeed, the arc $\{z_2^{(-2)}, z_4^{(2)}\}$ lies in all shifts of $\coaisle$ and intersects both module colimits $\{z_2,a_3\}$ and $\{z_4,z'_4\}$. Therefore, the only double fan sequence that is both Cauchy and compactly supported with respect to the coaisle metric $\cM_\coaisle$ is $\mathbf{F_1}$. Note that, although not a double fan sequence, the sequence $\mathbf{E}$ in the right-hand diagram in Figure~\ref{F:doublefanseq} is both Cauchy and compactly supported with respect to $\cM_\coaisle$.

    \subsection{Coaisle metric completion.} \label{SS:EX:completion}
    We may represent the metric completion of $\cC(\cZ)$ with respect to the coaisle metric $\cM_\coaisle$ combinatorially (cf.\ Theorem~\ref{T:completions agree}). In particular, the combinatorial completion $\fR_{(\cP,\rx)}$ is depicted by the right-hand diagram in Figure~\ref{F:completion}. The indecomposable objects of the category $\fR_{(\cP,\rx)}$ correspond to the arcs that lie in the cross hatched region where arcs can end at black accumulation points but not at white accumulation points. The left-hand diagram in Figure~\ref{F:completion} is the aisle $\aisle$. Notice the completion is simply the union of the shifts of the aisle where we now also allow some specific accumulation points to be endpoints of arcs (cf.\ Remark~\ref{R:completion_like_union_of_aisles}).

    \begin{figure}[H]
        \centering
        \includegraphics[scale=0.8]{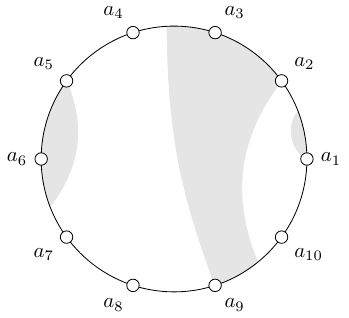}
        \;\;\;
        \includegraphics[scale=0.8]{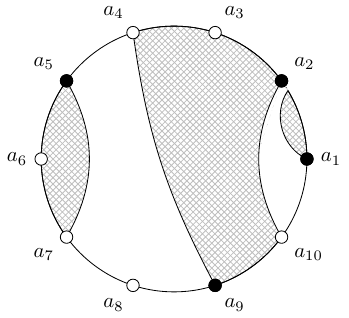}
        \caption{The aisle $\aisle$ of the t-structure $(\aisle,\coaisle)$ (left) and the completion of $\cC(\cZ)$ with respect to the coaisle metric $\cM_\coaisle$ (right).}
        \label{F:completion}
    \end{figure}

\bibliographystyle{alpha}
\bibliography{bib.bib}

\newcommand{\etalchar}[1]{$^{#1}$}
\begin{thebibliography}{BBDG18}

\bibitem[ACF{\etalchar{+}}23]{ACFGSII}
Jenny August, Man~Wai Cheung, Eleonore Faber, Sira Gratz, and Sibylle Schroll.
\newblock Cluster structures for the {$A_\infty$} singularity.
\newblock {\em J. Lond. Math. Soc. (2)}, 107(6):2121--2149, 2023.

\bibitem[Bal11]{Balmer_monad}
Paul Balmer.
\newblock Separability and triangulated categories.
\newblock {\em Adv. Math.}, 226(5):4352--4372, 2011.

\bibitem[BBDG18]{BBD}
Alexander Beilinson, Joseph Bernstein, Pierre Deligne, and Ofer Gabber.
\newblock {\em Faisceaux pervers}.
\newblock Soci{\'e}t{\'e} math{\'e}matique de France, 2018.

\bibitem[BCR{\etalchar{+}}24]{BCMPZ-completion}
Rudradip Biswas, Hongxing Chen, Kabeer~Manali Rahul, Chris~J. Parker, and Junhua Zheng.
\newblock Bounded $ t $-structures, finitistic dimensions, and singularity categories of triangulated categories.
\newblock {\em arXiv preprint arXiv:2401.00130v2}, 2024.

\bibitem[Bel00]{Beligiannis_relative}
Apostolos Beligiannis.
\newblock Relative homological algebra and purity in triangulated categories.
\newblock {\em J. Algebra}, 227(1):268--361, 2000.

\bibitem[BS21]{Balmer-Stevenson_relative}
Paul Balmer and Greg Stevenson.
\newblock Relative stable categories and birationality.
\newblock {\em J. Lond. Math. Soc. (2)}, 104(4):1765--1794, 2021.

\bibitem[{\c{C}}KP24]{CKP_discrete}
{\.I}lke {\c{C}}anak{\c{c}}{\i}, Martin Kalck, and Matthew Pressland.
\newblock Cluster categories for completed infinity-gons {I}: Categorifying triangulations.
\newblock {\em arXiv preprint arXiv:2401.08378v2}, 2024.

\bibitem[Fis17]{Fisher}
Thomas~Andrew Fisher.
\newblock {\em On the homological algebra of clusters, quivers, and triangulations}.
\newblock PhD thesis, Newcastle University, 2017.

\bibitem[Fra24]{Franchini_torsion_pairs}
Sofia Franchini.
\newblock Torsion pairs, t-structures, and co-t-structures for completions of discrete cluster categories.
\newblock {\em arXiv preprint arXiv:2403.08735}, 2024.

\bibitem[GHJ19]{GHJ}
Sira Gratz, Thorsten Holm, and Peter J{\o}rgensen.
\newblock Cluster tilting subcategories and torsion pairs in {I}gusa-{T}odorov cluster categories of {D}ynkin type {$A_\infty$}.
\newblock {\em Math. Z.}, 292(1-2):33--56, 2019.

\bibitem[GZ23]{Gratz--Zvonareva--2023--tstructures_in_clus_cat}
Sira Gratz and Alexandra Zvonareva.
\newblock Lattices of t-structures and thick subcategories for discrete cluster categories.
\newblock {\em J. Lond. Math. Soc. (2)}, 107(3):973--1001, 2023.

\bibitem[HJ12]{HJ-cat}
Thorsten Holm and Peter J{\o}rgensen.
\newblock On a cluster category of infinite {D}ynkin type, and the relation to triangulations of the infinity-gon.
\newblock {\em Math. Z.}, 270(1-2):277--295, 2012.

\bibitem[IT15]{ITcyclic}
Kiyoshi Igusa and Gordana Todorov.
\newblock Cluster categories coming from cyclic posets.
\newblock {\em Comm. Algebra}, 43(10):4367--4402, 2015.

\bibitem[Kel05]{Keller_orbit}
Bernhard Keller.
\newblock On triangulated orbit categories.
\newblock {\em Doc. Math.}, 10:551--581, 2005.

\bibitem[Kra20]{Krause_completing}
Henning Krause.
\newblock Completing perfect complexes.
\newblock {\em Math. Z.}, 296(3-4):1387--1427, 2020.
\newblock With appendices by Tobias Barthel and Bernhard Keller.

\bibitem[Law73]{lawvere1973metrics_on_categories}
F.~William Lawvere.
\newblock Metric spaces, generalized logic, and closed categories.
\newblock {\em Rend. Sem. Mat. Fis. Milano}, 43:135--166, 1973.

\bibitem[Mat24]{cyril}
Cyril Matou{\v{s}}ek.
\newblock Metric completions of triangulated categories from finite dimensional algebras.
\newblock {\em arXiv preprint arXiv:2409.01828}, 2024.

\bibitem[May01]{May-triangulated}
J.~P. May.
\newblock The additivity of traces in triangulated categories.
\newblock {\em Adv. Math.}, 163(1):34--73, 2001.

\bibitem[Mur23]{Murphy_Orlov}
Dave Murphy.
\newblock Bounding the orlov spectrum for a completion of discrete cluster categories.
\newblock {\em arXiv preprint arXiv:2308.01767}, 2023.

\bibitem[Mur25]{Murphy_K0}
Dave Murphy.
\newblock The grothendieck groups of discrete cluster categories of dynkin type ${A}_\infty$.
\newblock {\em Journal of Algebra}, 662:545--567, 2025.

\bibitem[Nee18]{Neeman--2018--DetermineEachOther}
Amnon Neeman.
\newblock The categories {${\mathcal T}^c$} and {${\mathcal T}^b_c$} determine each other.
\newblock {\em arXiv preprint arXiv:1806.06471}, 2018.

\bibitem[Nee20]{neeman2020metricssurvey}
Amnon Neeman.
\newblock Metrics on triangulated categories.
\newblock {\em Journal of Pure and Applied Algebra}, 224(4):106206, 2020.

\bibitem[Nee21]{neeman2018Brown}
Amnon Neeman.
\newblock Triangulated categories with a single compact generator and a {B}rown representability theorem.
\newblock {\em arXiv preprint arXiv:1804.02240v4}, 2021.

\bibitem[Ng11]{Ng}
Puiman Ng.
\newblock {\em Torsion theories and Auslander-Reiten sequences}.
\newblock PhD thesis, Newcastle University, 2011.

\bibitem[PY21]{Paquette--Yildirim--2021:combinatorial_completion}
Charles Paquette and Emine Y{\i}ld{\i}r{\i}m.
\newblock Completions of discrete cluster categories of type {$\mathbb A$}.
\newblock {\em Trans. London Math. Soc.}, 8(1):35--64, 2021.

\end{thebibliography}

\bigskip

\noindent 
\footnotesize \textsc{Charley Cummings, Department of Mathematics, Aarhus University, Ny Munkegade 118, 8000 Aarhus C, Denmark.}
\\
\noindent  \textit{Email address:} 
{\texttt{c.cummings@math.au.dk}}

\bigskip

\noindent 
\footnotesize \textsc{Sira Gratz, Department of Mathematics, Aarhus University, Ny Munkegade 118, 8000 Aarhus C, Denmark.}
\\
\noindent  \textit{Email address:} 
{\texttt{sira@math.au.dk}}

\end{document}